\documentclass[11pt]{amsart}
\usepackage[parfill]{parskip}    
\usepackage{fullpage}
\usepackage{graphicx}
\usepackage{tikz}
\usetikzlibrary{arrows,decorations.markings}
\usepackage{amssymb}
\usepackage{epstopdf}
\usepackage{ifpdf}
\usepackage{color}
\usepackage[all,cmtip]{xypic}
\usepackage{dsfont}
\definecolor{dmagenta}{rgb}{.5,0,.5} 
\definecolor{dred}{rgb}{.5,0,0} 
\definecolor{dgreen}{rgb}{0,.5,0} 
\definecolor{blue}{rgb}{0,0,0.5} 
\definecolor{black}{rgb}{0,0,0} 
\definecolor{vdgreen}{rgb}{0,.3,0} 
\definecolor{vdred}{rgb}{.3,0,0} 
\definecolor{red}{rgb}{1,0,0} 

\newcommand{\lplus}{{\mathfrak{h}}}        
\newcommand\cG{{\mathcal{G}}} 
\newcommand\cO{{\mathcal{O}}} 
\newcommand\cB{\mathcal{B}}    
\newcommand\CLO{{\mathcal{LO}}}    
\newcommand\AO{{A{\mathcal{O}}}}            
\newcommand\AOV{{A{\mathcal{O}_V}}}
\newcommand\IO{{I{\mathcal{O}}}}
\newcommand\IOV{{I{\mathcal{O}_V}}}

\newcommand{\F}{{\mathds{k}}}
\newcommand{\Z}{{\mathbb{Z}}}

\newcommand{\hairy}{\mathcal H}  

\newcommand{\phairy}{\mathcal{PH}} 
\newcommand{\G}{{\bf G}}              
\newcommand{\ext}{\bigwedge\nolimits}
\newcommand{\bdry}{\partial}

\newcommand{\id}{\mathrm{Id}}   

\newcommand{\comm}{{\sf  Comm}} 
\newcommand{\assoc}{{\sf  Assoc}} 
\newcommand{\lie}{{\sf  Lie}}            

\DeclareMathOperator{\Tr}{Tr}  
\DeclareMathOperator{\br}{\beta}  
\DeclareMathOperator{\SP}{Sp}  
\DeclareMathOperator{\Out}{Out}  
\DeclareMathOperator{\Sym}{{\frak S}}  
\DeclareMathOperator{\GL}{GL}  
\DeclareMathOperator{\Aut}{Aut}  

\DeclareMathOperator{\HD}{HD}
\DeclareMathOperator{\HC}{HC}
\let\ker\undefined
\DeclareMathOperator{\ker}{ker}  
\DeclareMathOperator{\spf}{\mathfrak{sp}}

\newcommand{\SF}[1]{{\mathbb S}_{#1}}  
\newcommand{\SpF}[1]{{\mathbb S}_{\langle#1\rangle}}  
\newcommand{\CB}{\mathsf{B}} 
\newcommand{\degree}[1]{[\![#1]\!]}
\newcommand{\Ok}[1]{\cO(\mkern-4.5mu(#1)\mkern-4.5mu)}
\newcommand{\pp}[1]{(\mkern-4.5mu(#1)\mkern-4.5mu)} 
\newcommand{\bfs}{{\mathbf x}} 
\newcommand{\iso}{\cong}  
\newcommand{\BLie}{\CB_\bullet{\lie_V}}
\newcommand{\BO}{\CB_\bullet\cO_V}
\newcommand{\BComm}{\CB_\bullet\comm_V}
\newcommand{\BAssoc}{\CB_\bullet\assoc_V}

\newtheorem{proposition}{Proposition}[section]
\newtheorem{definition}[proposition]{Definition}
\newtheorem{theorem}[proposition]{Theorem}
\newtheorem{lemma}[proposition]{Lemma}

\newtheorem{corollary}[proposition]{Corollary}
\theoremstyle{remark}

\newtheorem*{example}{Example}
\newtheoremstyle{red}{3pt}{3pt}{\color{red}}{}{\itshape}{.}{.5em}{}
\theoremstyle{red}

\newcommand{\directedwheel}[4]
{\begin{minipage}{3.5cm}
\begin{tikzpicture}[scale=.5]
\begin{scope}[decoration={markings,mark = at position 0.5 with {\arrow{stealth}}}] 
\node (v0) [above right] at (45:1.7) {$#1$};
\node (v1) [above left] at (135:1.7) {$#2$};
\node (v2) [below left] at (-135:1.7) {$#3$};
\node (v3) [below right] at (-45:1.7) {$#4$};
\draw (45:1)--(v0);
\draw (135:1)--(v1);
\draw (-135:1)--(v2);
\draw (-45:1)--(v3);
\midarrow (45:1) to [out=135, in=45] (135:1);
\midarrow (135:1) to [out=-135, in=135] (-135:1);
\midarrow (-135:1) to [out=-45, in=-135] (-45:1);
\midarrow (-45:1) to [out=45, in=-45] (45:1);
\end{scope}
\end{tikzpicture}
\end{minipage}
}

\newcommand{\midarrow}{\draw[postaction={decorate}]}

\title{Higher hairy graph homology}
\author{Jim Conant}
\author{Martin Kassabov}
\author{Karen Vogtmann}

\begin{document}
\begin{abstract}  
We study the hairy graph homology of a cyclic operad; in particular we show how to assemble corresponding hairy graph cohomology classes to form cocycles for ordinary graph homology, as defined by Kontsevich.  
We identify the part of hairy graph homology coming from graphs with cyclic fundamental group as the dihedral homology of a related associative algebra with involution.  For the operads $\comm, \assoc$ and $\lie$  we compute this algebra explicitly, enabling us to apply known results on dihedral homology to the computation of hairy graph homology.  In addition we determine the image in hairy graph homology of the trace map defined in \cite{CKV}, as a symplectic representation.  

For the operad $\lie$ assembling hairy graph cohomology classes yields all known non-trivial rational homology of $\Out(F_n)$. 
The hairy graph homology of $\lie$ is also useful for constructing elements of the cokernel of the Johnson homomomorphism of a once-punctured surface.


\end{abstract}
\maketitle

\label{sec:Introduction}


\section{Introduction}

In \cite{Ko2,Ko1} Kontsevich related the homology of certain infinite-dimensional graded Lie algebras to groups which are important in low-dimensional topology.  To construct one of these Lie algebras one needs a symplectic vector space $V$ and a cyclic operad $\cO$ with unit; the resulting Lie algebra $\CLO_V$ contains a copy of $\spf(V)$ in degree zero, which acts on the positive degree ideal $\lplus_V$ via the adjoint action.  Kontsevich considered  the limit $\lplus=\lim_{\dim(V)\to\infty} \lplus_V$ and showed, for example, that for  the associative operad $\assoc$ the $\spf_\infty$-invariants of $H_*(\lplus)$ can be identified with the cohomology of mapping class groups of punctured surfaces, and for the Lie operad $\lie$ with the cohomology of outer automorphism groups of free groups.  


Kontsevich's theorem is proved by identifying the subcomplex of symplectic invariants in the Chevalley-Eilenberg complex $\ext^*\lplus$   with a complex of {\em $\cO$-graphs}; these are finite graphs whose vertices are decorated by elements of $\cO$.    
In \cite{CKV} we introduced a modified version of $\cO$-graphs, called {\em hairy $\cO$-graphs} and a {\em trace map}   from $\ext^*\lplus$ to the hairy $\cO$-graph complex.   We proved that this trace map is injective on homology  and then used hairy graph homology methods in dimension one to obtain new information about the abelianization of $\lplus$.  

In \cite{Morita} Morita observed that elements  of the abelianization $H_1(\lplus)$ can be used to produce cocycles for $\lplus$ and thus, via Kontsevich's theorem, homology classes for $\Out(F_n)$ and mapping class groups.
In the current paper we  define an assembly map which produces cocycles for $\lplus$ from hairy graph homology classes of all dimensions, and show how this generalizes Morita's original construction. This provides new motivation for understanding hairy graph homology and its relation to the homology of $\lplus$.   Further motivation is provided by the fact that the Johnson cokernel of a once-punctured surface is a quotient  of $\lplus$ for $\cO=\lie$, and this cokernel maps onto a summand of $H_1(\lplus)$ that can be understood in terms of hairy graphs. The connection with the Johnson cokernel 
 is further explored in \cite{C, CK}, where new non-trivial elements of this cokernel are constructed.

We next turn to the problem of computing hairy graph homology $H_*(\hairy)$.  In \cite{CKV} we studied  $H_1(\hairy),$ determining it completely for $\cO=\comm$ and $\assoc$ and finding new classes  for $\cO=\lie$ which are related to modular forms.   In this paper we consider homology in all dimensions for the subcomplex of $\hairy$ spanned by connected graphs with cyclic fundamental group.  For the operads $\cO=\comm, \assoc$ and $\lie$ we prove that this homology can be identified with the dihedral homology (as defined by Loday) of a certain associative ring with involution $I\cO$.  For the operads $\comm, \assoc$ and $\lie$ we calculate $I\cO$ explicitly, then apply known results on dihedral homology to the computation of hairy graph homology.  For $\cO=\lie$ this gives us a complete answer, and for $\cO=\assoc$ and $\comm$ we obtain  low-dimensional calculations.

The paper is organized as follows.
 After  reviewing the basic constructions of hairy graph homology and the trace map in Section~\ref{sec:basics},  we  determine the image of the trace as a symplectic module  in Section~\ref{sec:image}. In particular we show that the irreducible representations in its decomposition appear with the same multiplicity as the corresponding general linear representations, a fact which is useful for understanding the new elements of the cokernel of the Johnson homomorphism mentioned above.   In Section~\ref{sec:assembly} we define the assembly map, which also has a nice interpretation in the language of modular operads that we sketch briefly.  In Sections~\ref{sec:BO} and \ref{sec:IO} we define the associative algebra $I\cO$ and do the calculations of dihedral homology for $\cO=\comm,\assoc$ and $\lie$.  
Finally, in Section~\ref{sec:implications} we briefly summarize the implications of this paper in the case  $\cO=\assoc$ for the homology of mapping class groups and $\cO=\lie$ for the homology of $\Out(F_n)$.

\section{Review of Kontsevich's theorem and the trace map}\label{sec:basics}
Throughout this paper we fix a field $\F$ of characteristic $0$.  All homology and cohomology is taken with trivial coefficients in $\F$ unless explicitly stated otherwise.

\subsection{Lie algebras}  
Given a cyclic  operad
 $\cO$ of finite-dimensional $\F$-vector spaces and a    symplectic vector space $(V, \langle\,,\rangle)$ we construct a Lie algebra $\CLO_V$ as follows.

The cyclic operad $\cO$ decomposes as $\cO=\oplus_s \Ok{s}$, where $\Ok{s}$ is a module over the symmetric group $\Sym_s$ spanned by elements with $s-1$ inputs and one output.  As a vector space, the Lie algebra $\CLO_V$ is defined as
$$\CLO_V=\bigoplus_{s\geq 2} [\Ok{s}\otimes V^{\otimes s}]^{\Sym_s},$$ 
where $\Sym_s$ acts on $V^{\otimes s}$ by permuting the factors.  Thus we can think of generators of $\CLO_V$ as operad generators whose i/o slots are no longer numbered, but are labeled by elements of $V$ instead.  These generators
 are called {\em spiders}, the underlying operad generator is called the {\em body} of the spider and the i/o slots are called its {\em legs}.  
 Two such generators are {\em mated} by choosing a leg from each, performing operad composition using the chosen leg on the first spider as the output slot and the second as the input slot, and multiplying the result by the symplectic product of the labels.   
   The Lie bracket on $\CLO_V$ is given by summing over all possible matings.  In other words, the Lie algebra structure is induced by the operations 
$$\Ok{s}\otimes\Ok{t}\otimes V^{\otimes 2}\to \Ok{s +t-2}\otimes \F,$$ where the map $\Ok{s}\otimes\Ok{t}\to \Ok{s+t-2}$ is  one of the compositions in the operad and $V^{\otimes 2}\to \F$ is the symplectic pairing of the output and input labels.  

We define the {\em degree} of an $s$-legged spider    
to be $s-2$. The positive degree elements of $\CLO_V$  generate a Lie subalgebra $\lplus_V=\bigoplus_{s\geq 3} [\Ok{s}\otimes V^{\otimes s}]^{\Sym_s}$. The  $2$-legged spiders generate a subalgebra which acts on $\lplus_V$ by the adjoint action. In the case that the operad has a unit element, this subalgebra contains a copy of $\mathfrak{sp}_V$  so we have an action of $\mathfrak{sp}_V$ on $\CLO_V$, and this restricts to an action on $\lplus_V$. This action coincides with the action of $\mathfrak{sp}_V$ induced by its natural action on $V^{\otimes s}$. 

\subsection{Graph complexes} In this sub-section we recall the definitions of  hairy $\cO$-graphs, the hairy graph complex $\hairy$, and hairy $\cO$-graph homology $H_*(\hairy)$.  To avoid technical difficulties we assume   $\Ok{2}\iso \F,$ spanned by the operad identity.  

\subsubsection{Graph terminology}  An \emph{admissible} graph is a finite $1$-dimensional CW complex, not necessarily connected, with no bivalent vertices. The univalent vertices are called {\em leaves} and the other vertices are called {\em internal vertices}. The edge adjacent to a leaf is called a {\em hair} and edges which are not hairs are {\em internal}.  An {\em orientation} of a graph is given by ordering   the internal vertices and directing the internal edges. Reversing an edge direction or applying a transposition to the vertex order reverses the graph orientation, so that every graph has two possible orientations. Finally, the \emph{degree} of a graph is equal to $\sum_{v}(|v|-2)$, where $v$ ranges over all internal vertices and $|v|$ is the valence of $v$.  

\subsubsection{Hairy graph complex}
A  vertex $v$ of a  graph is said to be {\em colored} by an operad element $o$ if the i/o slots of $o$ are identified with the half-edges adjacent to $v$.  An admissible graph  is said to be {\em colored} if it is oriented, all internal vertices are colored by operad elements and the leaves are ordered.  An \emph{orientation} of an admissible graph is defined by ordering the internal vertices and directing the internal edges, where swapping two vertices or reversing the direction of an edge give the opposite orientation.

Let $\mathcal{AC}_{k,s}$ denote the vector space spanned by all isomorphism classes of  oriented admissible colored graphs with $k$ internal vertices and $s$ leaves, modulo linear relations in $\cO$ and relations $(X,-or)=-(X,or)$.  The symmetric group $\Sym_s$ acts on $\mathcal{AC}_{k,s}$ by permuting the ordering of the leaves.  These are the $k$-chains of  a chain complex $\mathcal{AC}_s$ with boundary operator $\partial\colon\mathcal{AC}_{k,s}\to \mathcal{AC}_{k-1,s}$ defined by summing over contracting internal edges of the graph and using the element of $\cO$ at the beginning of the edge as input into the element at the end, with signs determined by orientation. In the case of the  Lie operad this chain complex  computes the cohomology of groups $\Gamma_{n,s}$ which generalize $\Out(F_n)$ and $\Aut(F_n)$ \cite[section 11]{CKV}.

Now define
$$C_k\hairy_{V}:=\bigoplus_{s\geq 0} [ \mathcal{AC}_{k,s}\otimes V^{\otimes s}]^{\Sym_s},$$
where $\Sym_s$ acts on $V^{\otimes s}$ by permuting the factors.  More informally, we have  attached vectors in  $V$ to the leaves, which are no longer ordered. The boundary operator on $\mathcal{AC}_s$ induces a boundary operator  $\partial\colon C_k\hairy_V\to C_{k-1}\hairy_{V}$.
Inside $\hairy_V$ there is a subcomplex   spanned by hairy graphs with no hairs.  Note that this is independent of $V$ and is the ordinary $\cO$-graph complex defined by Kontsevich.  

\subsection{Kontsevich's theorem}Let $V_n=\F^{2n}$ with the standard symplectic form.  The natural inclusions $V_n\to V_{n+1}$ induce maps of the Lie algebras $\CLO_{V_n}\to \CLO_{V_{n+1}}$, and we let   $\displaystyle\CLO=\lim_{n\to\infty}\CLO_{V_n}$.  The  (continuous) Lie algebra cohomology of $\CLO$  contains $H^*(\frak{sp}_\infty)$.  We can eliminate that summand by taking the $Sp$-invariant part of the  (continuous) Lie algebra cohomology of $\displaystyle\lplus=\lim_{n\to\infty}\lplus_{V_n}$.  Kontsevich's theorem can be then interpreted as saying (in part):  

\begin{theorem}\label{Kontsevich}   
The continuous Lie algebra cohomology of $\lplus$ has the structure of a Hopf algebra, whose primitive part we denote by $PH_c^*(\lplus)$.
\begin{enumerate} 
\item For $\cO$=\assoc, $\displaystyle PH_c^*(\lplus)^{Sp}\cong \bigoplus_{g\geq 0,s\geq 1} H_*(\operatorname{Mod}(g,s))$.
\item For $\cO$=\lie, $\displaystyle PH_c^*(\lplus)^{Sp}\cong \bigoplus_{n\geq 3} H_*(\Out(F_n))$.
\end{enumerate}
\end{theorem}

To prove this theorem,  Kontsevich first points out that the subcomplex of $Sp$-invariants in $\ext \CLO$ is quasi-isomorphic to the whole complex. 

He then uses invariant theory of the symplectic group to identify $Sp$-invariants in $\ext \CLO$ with $\cO$-graphs.  

Finally, he identifies the $\cO$-graph complex for $\cO=\assoc$ with the ``ribbon graph complex" computing $H^*(\operatorname{Mod}_{g,s})$ and the $\cO$-graph complex for $\cO=\lie$ with the cell complex of the moduli space of graphs, which computes $H^*(\Out(F_n))$.  See \cite{CV} for details.

\subsection{The trace map}  There is a natural inclusion of vector spaces $\iota\colon\ext\lplus\to\hairy$ that sends $\bfs= x_1\wedge\ldots\wedge x_k$ to the ordered disjoint union   $X= x_1\sqcup\ldots\sqcup x_k$, which is a hairy $\cO$-graph with $k$ (ordered) internal vertices and no internal edges.  There is a map $T\colon\hairy\to\hairy$ which sums over all ways of joining two leaves by an oriented edge and multiplying by the symplectic product of their labels.  The trace map  
$\Tr\colon \ext\lplus\to\hairy$ is then defined by $\Tr=\exp(T)\circ\iota,$
i.e. $$\Tr(\mathbf x)=X +T(X) + \frac{1}{2!}T^2(X)+ \frac{1}{3!}T^3(X)\ldots.$$
The sum is finite because $\mathbf x$ has only finitely many legs to pair.

In \cite{CKV} it is shown that $\Tr$ is a chain map and that it is injective on homology. A different proof that it is injective on homology will be presented below.

\subsection{Morita's construction of classes in  $(H^*(\lplus))^{Sp}$} Morita's idea was to use the map on Lie algebra cohomology induced by the abelianization map  $$\lplus\to \lplus^{\mathrm{ab}}$$ to pull back classes from $\ext^k\lplus^{\mathrm{ab}}=H^k(\lplus^{\mathrm{ab}})$, thereby obtaining potential classes in $H^k(\lplus)$ which can be projected to classes in $H^k(\lplus)^{Sp}$.  In order for this to work, one needs  to find elements of $\lplus^{\mathrm{ab}}$ to pull back, and this is where the trace map comes in.  The trace induces an injection $\Tr_*\colon H_*(\lplus)\to H_*(\hairy)$, so in particular, the abelianization $\lplus^{\mathrm{ab}}=H_1(\lplus)$ injects into $H_1(\hairy)$.  In \cite{CKV} we were able  to find interesting 1-dimensional hairy graph homology classes in the image of $\Tr_*$.  In the next section we determine  the entire image as an $\SP$-module.   

\section{The image of the trace map}\label{sec:image}

\subsection{Restricting to finite-dimensional pieces}
Recall that $\displaystyle\lplus=\lim_{n\to\infty}\lplus_{V_n}$ decomposes as a direct sum of chain complexes  $$\lplus =\bigoplus_{d\geq 3} \lplus \degree{d}$$ according to degree, and $\lplus_{V_n} \degree{d}$ is finite-dimensional for each $n$.  Similarly, $\displaystyle\hairy=\lim_{n\to\infty}\hairy_{V_n}$   decomposes as a direct sum according to degree.  Since the trace map $\Tr\colon \ext\lplus \to \hairy$ preserves both $n$ and degree,  we may restrict attention to a fixed $n$ and $d$,  find the image of the finite-dimensional piece $\ext^k\lplus_{V_n}\degree{d}$ in $C_k\hairy_{V_n}\degree{d}$, then analyze what happens as $n\to\infty$.   

Thus in this section we will fix $V$ of dimension  $2n$  together with a symplectic basis $$\cB=\{p_1,\ldots,p_n,q_1,\ldots,q_n\},$$ and let $x\mapsto x^*$ be the involution of $\cB$ defined by $p_i^*=q_i$ and $q_i^*=-p_i$.

\subsection{Decompositions of the degree $d$ parts of $\ext\lplus_V$ and $\hairy_V$}
The chain  complexes $\ext\lplus_V$ and  $\hairy_V$ are both modules over the symplectic group $\SP(V)$ and the trace map $\Tr$ is a module homomorphism.  The key to understanding the image of $\Tr$ is  to decompose the degree $d$ parts of both $\ext\lplus_V$ and  $\hairy_V$ as  direct sums of simpler $\SP(V)$-submodules.  These decompositions are based on the classical decomposition  of $V^{\otimes \ell}$ as $$V^{\otimes\ell}=V^{\langle\ell\rangle}\oplus V^{\langle \ell\rangle}_{\ell-2}\oplus\cdots\oplus V^{\langle \ell\rangle}_{\ell- 2p}$$
with $p=[\ell/2],$  and we begin by recalling this decomposition, as described in \cite[Section 17]{FH}.

\subsubsection{Classical decomposition of $V^{\otimes \ell}$}\label{classical} There are natural contraction operators $\Phi_{ij}\colon V^{\otimes \ell}\to V^{\otimes(\ell-2)}$, $1\leq i<j\leq \ell$ defined by
$$\Phi_{ij}(v_1\otimes\cdots\otimes v_\ell)=\langle v_i,v_j\rangle v_1\otimes\cdots\otimes\widehat{v_i}\otimes\cdots\otimes\widehat{v_j}\otimes\cdots\otimes v_\ell.$$

We define $V^{\langle\ell\rangle}\subset V^{\otimes \ell}$ by $$V^{\langle\ell\rangle}= \bigcap_{1\leq i<j\leq \ell} \ker(\Phi_{ij}).$$
For example, $V^{\langle 2\rangle}$ has a basis consisting of the tensors $x\otimes y$ for $x,y\in \cB$,  $y\neq x^*$ and $(p_1-x)\otimes (q_1+x^*)$ for $x\neq p_1$.  One can show that each $V^{\langle d\rangle}$ has a basis consisting of pure tensors $v_1\otimes\ldots\otimes v_d$.

There are also insertion operators $\Psi_{ij}\colon V^{\otimes (\ell-2)}\to V^{\otimes \ell}$, $1\leq i<j\leq \ell$ which ``insert the symplectic element" $\omega\in V^{\otimes 2}$ into the $i$th and $j$th spots in the tensor product.  Here $\omega$ is defined  by  $$\omega=\frac{1}{2n}\sum_{i=1}^n \, p_i\otimes q_i-q_i\otimes p_i =\frac{1}{2n}\sum_{x\in\cB}  \, x\otimes x^*.$$

Thus, for example  $$\Psi_{13}(v_1\otimes\cdots\otimes v_{\ell-2})
=\frac{1}{2n}\sum_{x\in\cB} x\otimes v_1\otimes x^*\otimes v_2\otimes\cdots\otimes v_{\ell-2}$$
and $\Psi_{12}$ can be written more concisely as $$\Psi_{12}(v_1\otimes\cdots\otimes v_{\ell-2})=\omega\otimes v_1\otimes\cdots\otimes v_{\ell-2}.$$

 We can graphically code the image under $\Psi_{ij}$ of a pure tensor $v_1\otimes\ldots\otimes v_{\ell-2}$  by plotting a row of $\ell$ dots, drawing an oriented red edge from $i$ to $j$ and labeling the remaining dots in order by the vectors $v_i$.
Thus the $2n$ terms of $\Psi_{13}(v_1\otimes\cdots\otimes v_{\ell-2})$ are represented by the single figure

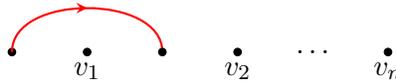
\begin{center}
\begin{tikzpicture}
\begin{scope}[decoration={markings,mark = at position 0.5 with {\arrow{stealth}}}]
 \draw [fill] (0,0) circle (.05);
 \draw [postaction={decorate}, red, thick] (0,0) to [out=90, in=90]  (2,0);
\draw [fill] (1,0) circle (.05);
\draw [fill] (2,0) circle (.05);
\node [below] at (1,0) {$v_1$};
\draw [fill] (3,0) circle (.05);
\node [below] at (3,0) {$v_2$};
\node (dots) at (4,0) {$\ldots$};
\draw [fill] (5,0) circle (.05);
\node [below] at (5,0) {$v_n$};
\end{scope} 
 \end{tikzpicture}
\end{center}

The subspace $V^{\langle \ell\rangle} _{\ell-2r}$ is now defined by
$$V^{\langle \ell\rangle} _{\ell-2r}=\sum_I\Psi_{I_1}\circ\cdots\circ \Psi_{I_r} (V^{\langle \ell-2r\rangle}),$$ where the sum is over all sets of pairwise disjoint pairs of indices $I=\{I_1,\ldots, I_r\}$.    Graphically, the image of a generator $v_1\otimes\ldots\otimes v_{\ell-2r}$ of $V^{\langle \ell-2r\rangle}$ is  represented by a row of $\ell$ dots with $r$ oriented red edges corresponding to the pairs $I_k$, and the rest of the dots labeled by the $v_i$.   For example, $\Psi_{14}\Psi_{36}((p_1+q_2)\otimes (q_1+p_2)\otimes p_3)$ is represented as follows:

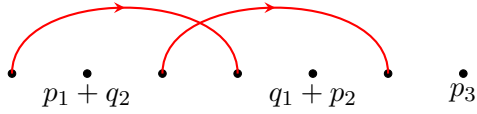
\begin{center}
 \begin{tikzpicture} 
 \begin{scope}[decoration={markings,mark = at position 0.5 with {\arrow{stealth}}}]
 \draw [fill] (0,0) circle (.05);
\draw [fill] (1,0) circle (.05);
\draw [fill] (2,0) circle (.05);
\draw [fill] (3,0) circle (.05);
\draw [fill] (4,0) circle (.05);
\draw [fill] (5,0) circle (.05);
\draw [fill] (6,0) circle (.05);
\node [below] at (1,0) {$p_1+q_2$};
\node [below] at (4,0) {$q_1+p_2$};
\node [below] at (6,0) {$p_3$};
\draw [postaction={decorate}, red, thick] (0,0) to [out=90, in=90]  (3,0);
\draw [postaction={decorate},, red, thick] (2,0) to [out=90, in=90]  (5,0);
\end{scope}
 \end{tikzpicture}
\end{center}

\begin{lemma}\label{Weyl}[Lemma 17.15, \cite{FH}] $V^{\otimes\ell}=V^{\langle\ell\rangle}\oplus V^{\langle \ell\rangle}_{\ell-2}\oplus\cdots\oplus V^{\langle \ell\rangle}_{\ell- 2p}$
with $p=[\ell/2].$
\end{lemma}

Note that the decomposition of $V^{\otimes\ell}$ in Lemma~\ref{Weyl}
is also invariant under the natural action of the symmetric group $\Sym_\ell$ on $V^{\otimes \ell}.$   
 
\begin{lemma}\label{onered} For each $i<j$, the projection $\pi\colon V^{\otimes\ell}\to V^{\langle \ell\rangle}$ is zero on  $\Psi_{ij}V^{\otimes (\ell-2)}$.  
\end{lemma}
\begin{proof} 
Let $t'=\Psi_{ij}(t)$ for some $t\in V^{\otimes (\ell-2)}$. Write $t= t_0 + t_1+\ldots + t_{\lfloor \frac{\ell-2}{2}\rfloor}$ using the decomposition of $V ^{\otimes (\ell-2)}$.  Then $\Psi_{ij}(t_i)\in V^{\langle \ell\rangle}_{\ell-2i+1}$ for all $i$, in particular no term of $\Psi_{ij}(t)$
is in $V^{\langle \ell \rangle}$. 
\end{proof}

The standard inclusion  $V_n\hookrightarrow V_{n+1}$ does not respect the decompositions in general, but it does preserve the first term, i.e.  $(V_n)^{\langle  \ell\rangle}\hookrightarrow (V_{n+1})^{\langle \ell\rangle}$.

\begin{lemma}\label{Nbig}
There exists  a finite list of elements $t_i\in (V_\ell)^{\langle \ell\rangle}$ such that for all $n\geq \ell$, the $t_i$  generate $(V_n)^{{\langle \ell\rangle}}$ as an $\SP(V_n)$-module.
\end{lemma}
\begin{proof}
$(V_n)^{{\langle \ell\rangle}}$ decomposes as a direct sum of finitely many Schur functors $\SpF{\lambda_i}(V_n)$, where $\SpF{\lambda_i}(V_n)$ is the irreducible $\SP(V_n)$-module corresponding to the partition $\lambda_i$.  (An explicit construction of this module is mentioned in the proof of Lemma~\ref{cor:sp-rep}.) The $\SpF{\lambda_i}(V_\ell)$ are non-trivial since $\ell$ is equal to the size of $\lambda_i$. Therefore we can pick a non-zero $t_i$ inside each  $\SpF{\lambda_i}(V_\ell),$ which embeds in $\SpF{\lambda_i}(V_n)$ as long as  $n\geq \ell$.
\end{proof}

 \subsubsection{Labeled objects}\label{red}  Let $t\in V^{\otimes \ell}$, and write $t$ as a sum of pure tensors $t=\sum_{i=1}^N v_{i1}\otimes\ldots\otimes v_{i\ell}$.  A {\em spider labeled by $t$} is a sum of $N$ spiders with identical bodies but different leg-labels:  the $j$th leg of the $i$th spider is labeled with $v_{ij}$.  If $t=\Psi_{ij}(t^\prime)$ for some $t^\prime\in V^{\otimes (\ell-2)}$ we refer to a spider labeled by $t$ as a spider with a {\em  red edge}, and indicate this graphically by drawing a single spider with one red edge from the $i$th leg to the $j$th leg, labeling the rest of the legs by  $t'$.  If $I=\{I_1,\ldots,I_r\}$ is a set of   disjoint pairs, then a spider labeled by an element in the image of $\Psi_I:=\Psi_{I_1}\circ\ldots\circ\Psi_{I_r}$ is a {\em spider with $r$ red edges}, indicated by drawing a single spider with red arrows between the appropriate legs.   We will be particularly interested in  spiders labeled by basis elements $t$ of 
 of $V^{\langle \ell\rangle}_{\ell-2r}$, i.e. $t=\Psi_I(t_0)$ for some $I=\{I_1,\ldots,I_r\}$ and $t_0\in V^{\langle \ell-2r\rangle}$.  
 
A wedge of spiders with a total of $\ell$ legs can similarly be labeled by $t\in V^{\otimes \ell}$; the result is a sum of $N$ wedges of spiders,   identical except for the labels. A {\em wedge of spiders with $r$ red edges} is a wedge labeled by an element $\Psi_I(t^\prime)=\Psi_{I_1}\circ\ldots\circ\Psi_{I_r}(t^\prime)$ for $t^\prime\in V^{\otimes \ell-2r}$.  Finally, a {\em hairy graph labeled by $t$} is a sum of $N$ hairy graphs with identical underlying graphs but different labels, and a {\em hairy graph with $r$ red edges} is a hairy graph labeled by an element in the image of $\Psi_I$, $I=\{I_1,\ldots,I_r\}$.  

\subsubsection{Decomposition of the degree $d$ part of $\ext^k\lplus_V$}  In degree $d$, a wedge $x_1\wedge\ldots\wedge x_k$ of $k$ spiders has a total of $\ell=d+2k$ labelled legs.  We decompose  the degree $d$ subcomplex of $\ext^k\lplus_V$
 by taking the tensor product of the labels and using the above (symmetric group invariant) decomposition of $V^{\otimes \ell}$ to obtain a decomposition
   $$(\ext^k\lplus_V)\degree{d}=(\ext^k\lplus_V)_{d,0}\oplus (\ext^k\lplus_V)_{d,1}\oplus\cdots\oplus (\ext^k\lplus_V)_{{d,\lfloor \ell/2}\rfloor},$$ 
where $(\ext^k\lplus_V)_{d,r}$ is spanned by wedges of spiders labeled by generators of $V^{\langle \ell\rangle}_{\ell-2r}$, i.e. by wedges of  $k$ spiders of total degree $d$ with $r$ red edges, whose remaining labels are in $V^{\langle \ell-2r\rangle}$.

\subsubsection{Decomposition of  the degree $d$ part of $\hairy_V$}  We can also use the classical decomposition  of tensor powers of $V$  to obtain a decomposition of $C_k\hairy_V\degree{d}$  as 
$$C_k\hairy_V\degree{d}=(C_k\hairy_V)_{d,0}\oplus (C_k\hairy_V)_{d,1}\oplus\cdots\oplus (C_k\hairy_V)_{d, \lfloor \frac{d}{2}\rfloor+k},$$ 
where $(C_k\hairy_V)_{d,r}$ is spanned by hairy graphs labeled by  generators of $V^{\langle s\rangle}_{s-2r}$ for some $s\leq d+2k$, i.e. by hairy graphs with $r$ red edges and the remaining labels in $V^{\langle s-2r\rangle}$.  

More precisely, recall that the $k$-chains $C_k\hairy_V$ are spanned by hairy graphs with $k$ internal vertices:

$$C_k\hairy_V =\bigoplus_{s\geq 0} [ \mathcal{AC}_{k,s}\otimes V^{\otimes s}]^{\Sym_s},$$

In degree $d$, each generator of $C_k\hairy_V$ corresponds to a graph with  $0\leq e \leq \lfloor{\frac{d}{2}}\rfloor + k$ oriented edges and $s=d+2k-2e$ hairs, with leaves labeled by elements of $V$.
In this notation, 
 $$(C_k\hairy_V)_{d,r}= \bigoplus_{s=0}^{d+2k} [ \mathcal{AC}_{k,s}\otimes  V^{\langle s\rangle}_{s-2r}]^{\Sym_s}.$$
Note that   $(\hairy_V)_{d,r}:= \oplus_k(C_k\hairy_V)_{d,r}$ is a subcomplex of $\hairy_V$, since the boundary map on $\hairy_V$ has no effect  on the hair labels.

\subsection{Approximating the trace map}

Recall that the trace map $\Tr\colon\ext \lplus\to \hairy$ is defined on a wedge of spiders $\bfs = x_1\wedge\ldots\wedge x_k$ by erasing the wedge signs and summing over all ways of connecting some pairs of spider legs by oriented edges, removing the  labels on these legs and multiplying by their symplectic products (the orientations on the edges determine the signs).   $\Tr$ is a chain map, and if we follow it by the projection $\pi\colon\hairy\to(\hairy)_0=\oplus_d(\hairy)_{d,0}$ we obtain a chain map 
$$\pi\circ\Tr\colon \ext \lplus\to (\hairy)_0.$$

Now  fix a $2n$-dimensional symplectic space $V,$ so that the notion of hairy graph with a ``red edge" is defined (see Section~\ref{red}).  

\begin{lemma}\label{lem:kerpi}
$ \pi \colon \hairy_V\to (\hairy_V)_0$ vanishes on hairy graphs with at least one red edge. 
 \end{lemma}
 \begin{proof}  
This follows immediately from Lemma~\ref{onered}.
 \end{proof}

We think of the oriented edges of a single hairy graph as ``black" edges.
There is  an obvious ``black-red" map 
$$\br\colon  (\hairy_V)_0\to\ext \lplus_V$$
defined on each generator of $(\hairy_V)_0$ by ``replacing all of the black edges with red edges"; thus if a hairy graph $X$ has degree $d$ and $e$ oriented edges, then $\beta(X)\in (\ext\lplus_V)_{d,e}$.  The ordering on the vertices of $X$ determines the ordering in the wedge product, and $\beta(X)$ comes with a coefficient of $\frac{1}{(2n)^e}$.  The map $\beta$  clearly surjects onto each $(\ext \lplus_V)_{d,e}$, since one can find a hairy graph in the pre-image of a generator by simply converting all of the red edges to black edges.  Since the degree $d$ part of $\ext \lplus_V$  is the direct sum of the $(\ext \lplus_V)_{d,e}$, and all of $\ext\lplus_V$ is the direct sum of its fixed degree subcomplexes, $\beta$ is surjective.

\begin{proposition} \label{lem:closetoid} Let $X$ be a generator of $(\hairy_V)_0$.  Then $\pi\circ\Tr\circ\br(X)$ is equal to $X$ plus a finite sum of  terms of the form $$\frac{1}{(2n)^i}Y$$
with $i\geq 1$.   The hairy graphs $Y$ depend on $X$ but not on $n$. 
\end{proposition}
\begin{proof}
Suppose $X$ has $e$ internal edges and $s$ hairs. Then $\br(X)$ is a wedge of spiders labeled by a generator  of $V^{\langle s+2e\rangle}_{s}$, i.e. by $\Psi_I(t)$ for some $I=\{I_1,\ldots,I_e\}$ and $t\in V^{\langle s \rangle}$.  Thus   $\br(X)$ has $(2n)^e$ terms, each with a coefficient of $\frac{1}{(2n)^e}$.    
We say the spider legs at either end of a red edge of $\beta(X)$ are {\em $X$-paired} and think of them as colored red, while other legs are colored black.

On each term of $\beta(X)$, the trace map $\Tr$ attaches black edges in all possible ways, then adds the results. 
 
 If we attach fewer than $e$ edges, then the terms of $\beta(X)$ can be grouped so that each group is a hairy graph with at least one red edge.  Since the projection $\pi$ is zero on such hairy graphs by Lemma~\ref{lem:kerpi}, we need only consider hairy graphs obtained by attaching exactly $e$ black edges.

If we attach $e$ black edges which coincide precisely with the $e$ red edges, we recover  a copy of $X$, but with a coefficient of $\frac{1}{(2n)^e}$.  Doing this for all terms of $\br(X)$ we obtain $(2n)^e$ such terms, giving a total contribution of  $X$ to $\Tr\circ\beta(X)$.

If we try to glue a black edge to two black legs we produce a coefficient of $0$, since the black legs are labeled by an element of  $V^{\langle s \rangle}$.

We can produce a non-trivial coefficient by gluing a black edge to two red legs that are not $X$-paired, but only if their labels pair non-trivially. This happens on only $(2n)^{e-1}$ terms of $\beta(X)$ (we have removed a  degree of freedom in the labelings); thus the coefficient of the resulting hairy graph $Y$ in $\Tr\circ\beta(X)$  will be $\frac{1}{(2n)^i}$ for some $i\geq 1$.   Similarly, we can connect  a black leg to a red leg  only if the labels pair non-trivially, and this also removes a degree of freedom.  
\end{proof}

\begin{corollary}\label{cor:iso} For $2n=\dim (V)$ large enough with respect to $d$, $\pi\circ\Tr\circ\br\colon (C_k\hairy_V)_{d,0}\to  (C_k\hairy_V )_{d,0}$ is an isomorphism.  
\end{corollary}
\begin{proof} Recall that for fixed $n$ and $d$,  the hairy graph $k$-chain space $C_k\hairy_V\degree{d}$ is finite-dimensional, so the subspace $(C_k\hairy_V)_{d,0}$ is also  finite-dimensional and we can apply Lemma~\ref{Nbig}  to conclude that for sufficiently large $n,$ $(C_k\hairy_V)_{d,0}$ is finitely generated as an $\SP(V)$-module  by a fixed number of generators which does not change as $n$ increases.  Using this basis,  we can represent $\pi\circ\Tr\circ \br\colon (C_k\hairy_V)_{d,0}\to (C_k\hairy_V)_{d,0}$ by a matrix of fixed size with entries in the group ring $\F[\SP(V)]$.
By Lemma~\ref{lem:closetoid}, $\pi\circ\Tr\circ \br=\mathrm{Id}+A(n)$ where 
$\displaystyle\lim_{n\to\infty} A(n)=0$. Since invertibility is an open condition,
$\pi\circ\Tr\circ \br$ is invertible on $(C_k\hairy_V)_{d,0}$ for large enough $n$. 
\end{proof}  

\begin{corollary} For $2n=\dim (V)$ large enough with respect to $d$, $\pi\circ\Tr\colon (\ext^k\lplus_V)\degree{d}\to (C_k\hairy_V)_{d,0}$ is an isomorphism.
\end{corollary}
\begin{proof} By Corollary~\ref{cor:iso}, the composition $\pi\circ\Tr\circ\beta$ is an isomorphism, showing that $\beta$ is injective.  Since we already know that $\beta$ is surjective, we have that $\beta$ is an isomorphism, and thus $\pi\circ\Tr$ is also an isomorphism.  
\end{proof}

Note that although $\Tr$ is compatible with the inclusions $V_n\hookrightarrow V_{n+1}$, the projection $\pi$ is not, so the above isomorphisms do not induce a map on the direct limit.  Nevertheless, since the chain complexes which compute homology in each degree are isomorphic for $\dim V$ sufficiently large we do have:

\begin{corollary} Fix a degree $d$.  Then there is an $N$ such that if $\dim V>N$,  
$\pi\circ\Tr$ induces an
 isomorphism $ H_k(\lplus_V)\degree{d}\cong H_k((\hairy_V)_0)\degree{d}$.
\end{corollary}

We note that surjectivity of $\pi\circ\Tr$ follows also from Theorem 3.6 of \cite{CKV}. since the image of the map which is proved to be surjective in Theorem 3.6 generates $H_k((\hairy_V)_0)\degree{d}$ as a symplectic module.

Finally, we have the following statement about the decomposition of hairy graph homology into irreducible representations.

\begin{corollary}\label{cor:sp-rep}
Fix a degree $d$, and let 
$$H_k(\hairy_V)\degree{d}\cong \bigoplus_{\lambda\in\Lambda} \SF{\lambda}(V)$$  be the decomposition of   $H_k(\hairy_V)\degree{d}$
 into $\GL(V)$-Schur functors, where $\Lambda$ is some index set. Then its $\SP(V)$-decomposition is
$$H_k(\lplus_V)\degree{d}\cong \bigoplus_{\lambda\in\Lambda} \SF{\langle\lambda\rangle} (V)$$ for the same index set $\Lambda$.
\end{corollary}
\begin{proof}
Recall 
$$C_k\hairy_V =\bigoplus_{s\geq 0} [ \mathcal{AC}_{k,s}\otimes V^{\otimes s}]^{\Sym_s}\cong
\bigoplus_{s\geq 0}\mathcal{AC}_{k,s}\otimes_{\Sym_s} V^{\otimes s}.$$
Observe that the boundary maps on $C_k\hairy_V$ and $C_k(\hairy_V)_0$ are induced by a $\Sym_s$-equivariant boundary map on $\oplus \mathcal{AC}_{k,s}$, so we have that 
$$H_k(\hairy_V)=\oplus_{s\geq 0}H_k(\mathcal{AC})\otimes_{\Sym_s}V^{\otimes s}$$
and
$$H_k((\hairy_V)_0)=\oplus_{s\geq 0}H_k(\mathcal{AC})\otimes_{\Sym_s}V^{\langle s\rangle}. $$

Suppose that $H_k(\mathcal{AC})=\oplus_{\lambda\in \Lambda} P_{\lambda}$, where $P_\lambda$ is the irreducible representation of $\Sym_{s}$ corresponding to the partition $\lambda$. Then, since $\SF{\lambda}(V)$ is defined to be $P_\lambda\otimes_{\Sym_s}V^{\otimes s}$, we have that $H_k(\hairy)\cong \oplus_{\lambda \in \Lambda} \SF{\lambda}(V)$. Similarly, by \cite[section 17.3]{FH} we have that $\SpF{\lambda}(V)=P_\lambda\otimes_{\Sym_s}V^{\langle s\rangle}$, and so
 $H_k(\lplus_V)\cong H_k((\hairy_V)_0)=\oplus_{\lambda\in \Lambda} \SpF{\lambda}(V)$.    
\end{proof}

\begin{example} In this section we have decomposed the hairy graph complex according to the degree of hairy graphs.  The differential also respects the rank of the fundamental group, and we can also decompose according to this rank.  Since the trace map embeds $H_*(\lplus_V) $ into the homology of the hairy graph complex, we get an induced grading on the homology of $\lplus_V$.  In particular, the abelianization can be decomposed as
$$\lplus_V^{\mathrm{ab}}=\bigoplus_{r\geq 0}(\lplus_V^{\mathrm{ab}})_{(r)}.$$ 
These are symplectic modules, and for the case $\cO=\lie$ it is easy to compute their decomposition into irreducibles in ranks 0 and 1 using hairy graph techniques   as  
$$(\lplus_V^{\mathrm{ab}})_{(0)}=\SpF{[1^3]}(V) , \,\,\,\,\, (\lplus_V^{\mathrm{ab}})_{(1)}=\bigoplus_{k\geq 0}\SpF{[2k+1]}(V),$$
 where as before $\SpF{\lambda}(V)$ denotes the irreducible $\SP(V)$-module corresponding to the partition $\lambda$.
These pieces of the abelianization were found by Morita, who constructed a surjective map onto them using his original trace map in\cite{Morita}.  In \cite{CKV} we found rank two hairy graph homology classes in the image of our trace map, thus supplying new elements of  $(\lplus_V^{\mathrm{ab}})_{(2)}$.  Corollary~\ref{cor:sp-rep} can now be used together with the computations from \cite{CKV}  to determine the rank two piece of $\lplus^{\mathrm{ab}}$ explicitly as a symplectic module.  We obtain
 $$(\lplus_V^{\mathrm{ab}})_{(2)}= \bigoplus_{k>\ell\geq 0}\left(\SpF{[2k,2\ell]}(V)\otimes \mathcal S_{2k-2\ell+2}\right)\oplus \left(\SpF{[2k+1,2\ell+1]}(V)\otimes\mathcal M_{2k-2\ell+2}\right)$$
 where $\mathcal S_{w}$ and $\mathcal M_w$ are the vector spaces of classical weight $w$ cusp forms and modular forms respectively. 
 Hand calculations indicate that $(\lplus_V^{\mathrm{ab}})_{(3)}$ is also nonzero and related to modular forms, though a complete calculation has not yet been done. The higher rank pieces are unknown, but there is no reason to expect them to be trivial.
\end{example}

\section{From  the hairy graph complex to the  the graph complex}\label{sec:assembly}

In this section, we show how to use hairy graph homology classes in all dimensions to produce ordinary graph cohomology classes in higher dimensions.  Applying this to the Lie and associative operads gives potential new homology classes for $\Out(F_n)$, $\Aut(F_n)$ and $\operatorname{Mod}(g,n)$.
In the last part of this section we sketch a translation of what we have done into the language of modular operads.
 
\subsection{Assembly map}\label{assembly}

We first define a map $S \colon   \hairy_V\to\hairy_V$  which  
``splits" edges of a hairy graph one at a time, then adds.  Given a hairy graph $X$, an oriented edge $e\in X$ and an element $x\in\cB$, we can form a new hairy graph $X_{e,x}$ by  cutting  $e,$ labeling the initial
cut point with $x$ and the terminal cut point with its dual $x^*$.    Then
$$S(X)=\frac{1}{\dim(V)}\sum_{e,x} X_{e,x}$$
where the sum is 
over all oriented edges $e$ of $X$ and basis elements $x\in \cB\ $(see Figure~\ref{split} for an example).
If $X$ has $k$ oriented edges and $\dim(V)=2n$, then $S(X)$ has $2nk$ terms.

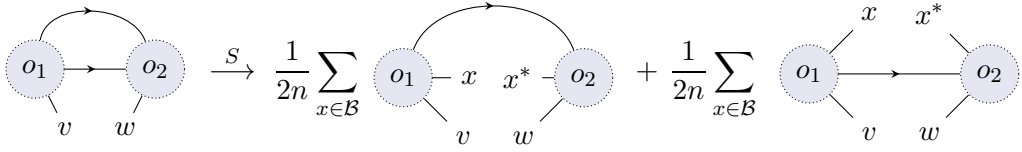
\begin{figure}
$$
\begin{minipage}{2.6cm}
\begin{tikzpicture}[scale=.8]
\tikzstyle{operad}=[draw,shape=circle, densely dotted, fill=blue!10];
\begin{scope}[decoration={markings,mark = at position 0.5 with {\arrow{stealth}}}]
\node [operad] (left) at (-1,0) {$o_1$};
\node [operad] (right) at (1,0) {$o_2$};
\node (v) at (-.5, -1) {$v$};
\node (w) at (.5,-1) {$w$};
\midarrow (left) -- (right);
\midarrow (left) to [out=80, in=110] (right);
 \draw (v) -- (left);
 \draw (w) -- (right);
 \end{scope}
 \end{tikzpicture}
 \end{minipage}
\buildrel S\over\longrightarrow\,\, \frac{1}{2n}\sum_{x\in\mathcal B}\,\,
 \begin{minipage}{3.4cm}
 \begin{tikzpicture}[scale=.8]
 \tikzstyle{operad}=[draw,shape=circle, densely dotted, fill=blue!10];
 \begin{scope}[decoration={markings,mark = at position 0.5 with {\arrow{stealth}}}]
\node [operad] (left) at (0,0) {$o_1$};
\node [operad] (right) at (3,0) {$o_2$};
\node (v) at (1, -1) {$v$};
\node (w) at (2,-1) {$w$};
\node (x) at (1.1,0) {$ x $};
\node (xstar) at (1.9,0) {$ x^\ast $};
 \draw  (left) to (x);
 \draw (xstar) -- (right);
\midarrow (left) to [out=70, in=110] (right);
 \draw (v) -- (left);
 \draw (w) -- (right);
 \end{scope}
\end{tikzpicture}
\end{minipage}
+\,  \frac{1}{2n}\displaystyle\sum_{x\in\mathcal B}\ \ 
\begin{minipage}{4cm}
 \begin{tikzpicture}[scale=.8]
 \tikzstyle{operad}=[draw,shape=circle, densely dotted, fill=blue!10];
 \begin{scope}[decoration={markings,mark = at position 0.5 with {\arrow{stealth}}}]
\node [operad] (left) at (0,0) {$o_1$};
\node [operad] (right) at (3,0) {$o_2$};
\node (v) at (1, -1) {$v$};
\node (w) at (2,-1) {$w$};
\midarrow (left) -- (right);
 \node (x) at (1,1) {$ x $};
\node (xstar) at (2,1) {$ x^\ast $};
 \draw  (left) to (x);
 \draw (xstar) -- (right);
 \draw (v) -- (left);
 \draw (w) -- (right);
 \end{scope}
\end{tikzpicture}
\end{minipage}
$$
  \caption{The map $S$}\label{split}
\end{figure}

It is easy to check that $S$ is a chain map, so the exponential $e^S\colon \hairy_V\to\hairy_V$ defined by
$$e^S(X) = X + S(X) + \frac{1}{2}S^2(X) + \frac{1}{3!}S^3(X)+\ldots$$
is also a chain map.  Note that the sum is finite, since $X$ has only a finite number of edges that can be split.

 The primitive part $\mathcal{PG}$ of Kontsevich's graph complex sits as a subcomplex of $\hairy_V$ for any $V$, namely as the subcomplex spanned by connected graphs with no hairs.   Splitting any edge of such a graph produces an $\SP$-invariant in $\hairy_V$, so restricting $e^S$ to $\mathcal{PG}$ gives a chain map $$\mathcal{PG}\to(\hairy_V)^{\SP}\subset \hairy_V.$$

The chain complexes $\mathcal{PG}, (\hairy_V)^{\SP}$ and $\hairy_V$ each decompose as a direct sum of subcomplexes according to the degree of graphs.   In each degree, the $k$-chains are  finite-dimensional. Since $e^S$ preserves degree it restricts to each subcomplex. Taking the  dual in each degree of the induced maps  $H_k(\mathcal{PG})\to H_k(\hairy_V)^{\SP}\to H_k(\hairy_V) $  gives maps on the degree-graded dual
\begin{equation}\label{chainlevel}(H_*(\hairy_V))^\dagger\to (H_*(\hairy_V)^{\SP})^\dagger\to H_*(\mathcal{PG})^\dagger.
\end{equation}

We now work on further identifying the first two spaces.  
 
 Given any graded vector space $U$, let $\F[U]$ denote the free graded commutative algebra on $U$, where graded commutativity means that $xy=(-1)^{|x||y|}yx$. Denote by $\mathcal{PH}_V\subset \hairy_V$ the subspace spanned by connected hairy graphs. By decomposing a graph into its connected components and grading by the number of internal vertices we get an isomorphism $\hairy_V\cong \F[\mathcal{PH}_V]$. 
Since the $\SP$-action commutes with the boundary map, we get both
 \[
 H_*(\hairy_V)=H_*(\F[\mathcal{PH}_V]) \qquad\hbox{and}\qquad  H_*(\hairy_V)^{\SP}=H_*(\F[\mathcal{PH}_V])^{\SP}.
 \]
 By a theorem of Milnor-Moore \cite{MM65}, any graded commutative cocommutative Hopf algebra over a field of characteristic zero is a graded polynomial algebra on the space of primitive elements.
Taking primitives commutes with homology, so we have
\[
H_*(\F[\mathcal{PH}_V])=\F[H_*(\mathcal{PH}_V)] \qquad\hbox{and}\qquad H_*(\F[\mathcal{PH}_V])^{\SP}=\F[H_*(\mathcal{PH}_V)]^{\SP}.
\]

Taking graded duals (graded by degree) and applying the Universal Coefficient theorem gives 
 \[
\F[H_*(\mathcal{PH}_V)]^\dagger=\F[H_*(\mathcal{PH}_V)^\dagger]=\F[H^*_c(\mathcal{PH}_V)]   
\]
\centerline{ and }
\[(\F[H_*(\mathcal{PH}_V)]^{\SP})^\dagger=(\F[H_*(\mathcal{PH}_V)^\dagger])_{\SP}=\F[H^*_c(\mathcal{PH}_V)]_{\SP}.
\]
Thus Equation~\ref{chainlevel} becomes the following ``assembly map" 
\begin{equation}
 \F[H^*_c(\mathcal{PH}_V)]\to \F[H^*_c(\mathcal{PH}_V)]_{\SP}\to H_*(\mathcal{PG})^\dagger =H_c^*(\mathcal{PG}).
\end{equation}

In other words, we have an explicit recipe for assembling hairy graph cohomology classes into $\cO$-graph cohomology classes, as follows: 
\begin{itemize} 
\item Represent cohomology classes in $H^*_c(\mathcal{PH}_V)$ by  cocycles $\zeta_1,\ldots,\zeta_m$ defined on $\mathcal{PH}_V$.  
 \item Let $\zeta$ be the image of the monomial $\zeta_1\ldots\zeta_m$ in $\F[H^*_c(\mathcal{PH}_V)]_{\SP}$.  
 \item To evaluate  $\zeta$ on a generator of $\mathcal{PG}$, i.e. on a connected $\cO$-graph $\G,$ first apply $e^S$ to break $\G$ into a sum of connected hairy graphs in all possible ways. The value of $\zeta$ is non-zero only on terms with exactly $m$ connected components $X_1,\ldots,X_m$, in which case its value is  $$\sum_{\sigma\in\Sym_m} \epsilon(\sigma) \zeta_{1}(X_{\sigma(1)})\ldots  \zeta_m(X_{\sigma(m)}),$$
 where the (Koszul) sign $\epsilon (\sigma)$ is determined by the equation $X_1\cdots X_k=\epsilon (\sigma) X_{\sigma(1)}\cdots X_{\sigma(k)}$.
\end{itemize}
 
\begin{example} Let $\cO=\lie$. 
Then the first homology $H_1(\mathcal{PH}_V)$ has a summand isomorphic to $S^3V$, where a generator $v_1 v_2 v_3$ is represented by the hairy graph depicted below (see \cite{CKV} for a detailed explanation of this construction).

$$ 
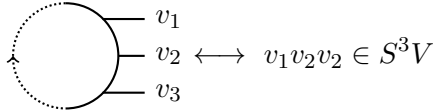
\begin{minipage}{2.3cm}
\begin{tikzpicture} [thick, scale=.7]
\draw[densely dotted,<-] (-1,0) arc (180:270:1);
\draw[densely dotted] (0,1) arc (90:180:1);
\draw (0,-1) arc (-90:90:1);
\draw (1,0) -- (1.5,0) node [right] {$v_2$};
\draw (.7,.7) -- (1.5,.7) node [right] {$v_1$};
\draw (.7,-.7) -- (1.5,-.7) node [right] {$v_3$};
\end{tikzpicture}
\end{minipage}
\longleftrightarrow\,\,
v_1v_2v_2\in S^3V
$$

Let  $\zeta_1$ be the cocycle which evaluates to $1$ on the above hairy graph if  $v_1=v_2=v_3=p_1$ and is zero on any other connected hairy graph. Similarly define $\zeta_2$ to be $1$ if $v_1=v_2=v_3=q_1$  and is $0$ elsewhere.
 Let $\G$ be an $\cO$-graph consisting of two loops with three edges joining them:
$$ 
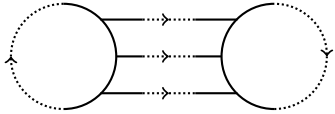
\begin{tikzpicture} [thick, scale=.7]
\draw[densely dotted,<-] (-1,0) arc (180:270:1);
\draw[densely dotted] (0,1) arc (90:180:1);
\draw (0,-1) arc (-90:90:1);
\draw (4,1) arc (90:270:1);
\draw[densely dotted] (4,-1) arc (-90:0:1);
\draw[densely dotted,<-] (5,0) arc (0:90:1);
\draw (1,0) -- (1.5,0);
\draw (.7,.7) -- (1.5,.7);
\draw (.7,-.7) -- (1.5,-.7);
\draw (3,0) -- (2.5,0);
\draw (3.3,.7) -- (2.5,.7);
\draw (3.3,-.7) -- (2.5,-.7);
\draw[densely dotted,->] (1.5,0)--(2,0);
\draw[densely dotted] (2,0)--(2.5,0);
\draw[densely dotted,->] (1.5,.7)--(2,.7);
\draw[densely dotted] (2,.7)--(2.5,.7);
\draw[densely dotted,->] (1.5,-.7)--(2,-.7);
\draw[densely dotted] (2,-.7)--(2.5,-.7);
\end{tikzpicture} 
 $$
 When we evaluate $\zeta_1\zeta_2$ on $e^S(\G)$,  the only terms that will contribute are those where all three joining edges are snipped, one of the connected components  has all labels equal to $p_1$ and the other has all labels equal to $q_1$. Therefore  evaluating $\zeta_1\zeta_2=-\zeta_2\zeta_1$ on $e^S(\G)$ gives $\frac{1}{(2n)^3} +\frac{1}{(2n)^3}$.

Recall that Kontsevich's theorem~\ref{Kontsevich} for $\cO=\lie$ gives
\[
H_c^k(\mathcal{PG})\degree{d}\cong H_{d-k}(\Out(F_{\frac{d}{2} + 1})).
\]
 The cocycle we have just constructed corresponds to the first Morita class $\mu_1\in H_4(\Out(F_4)$.   
We give many other examples for $\cO=\lie$ in \cite{CHKV}.  For very small values of $n$ many of the cocycles constructed in this way  for $\Out(F_n)$  have been shown to be non-trivial in homology.  
\end{example}

The assembly map basically reformulates the question of whether cocycles produced in this way are non-trivial  as a question about understanding the behavior of the map $e^S\colon \mathcal{PG}\to \hairy$  on homology.  

Our assembly map is closely related to the Feynman transform defined by Getzler and Kapranov in their theory of modular operads \cite{GK}.  We outline this connection below, but  first we pause to compare the construction of cocycles via the assembly map with the abelianization construction from \cite{CKV}.

\subsection{Comparison to the abelianization construction}
Let $A\colon ( \ext H^1_c(\hairy))^{\SP}\to H^*(\mathcal G)$ denote the assembly map restricted to first cohomology. In   \cite{CKV} we produced $\SP$-invariant cohomology classes for $\lplus$, and therefore for $\Out(F_n),$ from $\ext H^1_c(\hairy)$  by applying the cohomology and symplectic invariant functors to the chain of maps $\lplus\to \lplus^{{\rm ab}}\to H_1(\hairy)$, where $\lplus_{{\rm ab}}$ and $H_1(\hairy)$ are regarded as abelian Lie algebras.  Let $B\colon (\ext H^1_c(\mathcal H))^{\SP}\to H^*_c(\mathfrak h)^{\SP}$ denote this map.
In this section we show that the maps $A$ and $B$ give the same information. Thus the assembly map is a strict generalization of the abelianization construction.

Recall from \cite{CV} that there is a chain map $\psi\colon \ext \mathfrak h \to \cG$ which can be thought of, in the current context, as the part of $\Tr$ that lands in graphs with no hairs. It is defined as the sum over all ways of gluing in black edges along a complete matching of the legs in a wedge of spiders.
 For $V$ of large enough dimension compared to $d$, it induces an isomorphism $\ext(\lplus)^{\SP}\degree{d}\overset{\cong}{\longrightarrow} C_*(\cG)\degree{d}$.  In order to state the following proposition, we define an automorphism $\hat{\psi}\colon ( \ext H_1(\hairy))^{\SP}\degree{d}\to ( \ext H_1(\hairy))^{\SP}\degree{d}$ as follows.   
 It is defined, like $\psi,$ by summing over gluing edges into a matching of the legs of a wedge of hairy homology classes, but the glued-in edges are red edges in the sense of section 3.
 Like $\psi$, this is an isomorphism for $\dim(V)$ large enough compared to $d$.

\begin{proposition}
The following diagram commutes.
$$
\xymatrix{
\ext(\lplus)^{\SP} \ar@{>->>}[d]^{\psi_*} \ar[r]&\ext(\lplus_{\rm ab})^{\SP}\ar[r]&(\ext H_1(\hairy))^{\SP}\ar[d]^{\hat{\psi}}\\
C_*(\cG)\ar[r]^{\!\!\!\!\!\exp S_*}&\F[C_*(\mathcal{PH})]^{\SP}\ar@{->>}[r]&( \ext H_1(\hairy))^{\SP}
}
$$
where the map $\F[C_*(\mathcal{PH})]\to \ext H_1(\hairy)$ is defined by first projecting $\F[C_*(\mathcal{PH})]\to \F[C_1(\mathcal {PH})]=\ext C_1(\mathcal {H})$ and then projecting $C_1(\mathcal H)\twoheadrightarrow H_1(\mathcal H)$.
\end{proposition}
\begin{proof}
The automorphism $\hat{\psi}$ mimics the effect of $\exp(S)\circ\psi_*$.
\end{proof}
\begin{corollary}
We have $B\hat\psi^*=\psi^*A.$ Hence, since $\psi^*$ and $\hat\psi^*$ are chain isomorphisms for  each fixed degree $d$ and $\dim V$ large enough, the maps $A$ and $B$ give the same information.
\end{corollary}
\begin{proof}
Apply the cohomology functor to the above commutative diagram.
\end{proof}

\subsection{Cocycles via the Feynman transform}
The assembly map has an easy conceptual explanation using the language of modular operads introduced by Getzler and Kapranov~\cite{GK}, which leads to a construction of assembly map without reference to the vector space $V$. This points the way to a topological interpretation of the assembly map in the case $\cO=\lie$    which we  explore further in \cite{CHKV}.

Whereas a cyclic operad is an operad in which one regards all i/o slots as equivalent, a modular operad is a cyclic operad with an additional grading by \emph{genus}, $\mathcal M\pp{n}=\bigoplus_{g\geq 0}\mathcal{M}\pp{n,g},$ and an i/o slot may be plugged into another i/o slot on the same element. This is formalized by giving {\em contraction operators} $\xi_{i,j}\colon \mathcal{M}\pp{n,g}\to  \mathcal{M}\pp{n-2,g+1}$ that satisfy natural  axioms.
A slight generalization of this is the notion of \emph{twisted modular operads} where the axioms involving the contraction operators contain additional signs which are governed by a \emph{hyperoperad} $\mathfrak D$, see~\cite{GK} for details. A further extension is \emph{differentially graded  modular operads}  where the spaces $\mathcal{M}\pp{n,g}$ are complexes and the differential satisfies additional compatibility axioms, which gives a modular operad structure on the (co)homology.

Given a cyclic operad $\cO$ (or more generally a stable module over the symmetric group), there is natural {\em modularization} of $\cO,$ denoted $\mathbb{M}_{\mathfrak D}\cO$.  A generator of $\mathbb M_{\mathfrak D}\cO$  is similar to a hairy graph except that the hairs are numbered instead of being labeled by elements of a vector space $V$.  In this construction the hyperoperad ${\mathfrak D}$ controls how the orientation data of the underlying graph is used to construct the equivalence relation.

One of the main results in~\cite{GK} is the dual of $\mathbb M_{\mathfrak D}\cO$, which is called the Feynman transform $\mathsf{F}_{\mathfrak{D}}\cO$  of $\cO$, has a natural structure of differential graded twisted modular operad where the signs are determined by a related hyperoperad $\mathfrak{D}^\vee$.  

With a suitable choice of the hyperoperad $\mathfrak{D}$,  the space of connected hairy graphs can be identified with 
$$
\bigoplus_n \left[\mathbb M_{\mathfrak D}\cO\pp{n} \otimes V^{\otimes n} \right]^{\Sym n}
$$ 
and the hairy graph differential is dual to the differential of the  Fyenman transform $\mathsf{F}_{\mathfrak{D}}\cO$. Thus the dual of hairy graph homology can be obtained from the twisted modular operad $\mathsf{F}_{\mathfrak{D}}\cO$, i.e.,
$$
\bigoplus_{g,n} \left[ H_*(\mathsf{F}_{\mathcal{D}}\cO)\pp{g,n}\otimes V^{\otimes n} \right]^{\Sym_n} \cong H^*(\phairy\cO_V).
$$
The modular operad structure gives a composition map from the modularization
$$
\mathbb{M}_{\mathfrak{D}^\vee} H_*(\mathsf{F}_{\mathfrak{D}}\cO)
\to 
H_*(\mathsf{F}_{\mathfrak{D}}\cO)
$$
obtained by using the composition/contraction operations along the edges of the graph. The assembly map described in \S\ref{assembly} can be obtained from this map by attaching labels to the i/o slots of the operad and using the above identification with  $H^*(\phairy\cO_V)$.

\section{Haircuts}
\label{sec:BO}
%
%
%

A graph with no univalent vertices is called a {\em core graph}.  Every finite graph $G$ can be reduced to a core graph $c(G)$ by removing all univalent vertices (and the interiors of the adjacent edges),  then repeating until there are no univalent vertices left.  If $G$ is a union of trees, then $c(G)$ is the empty graph.  The components of $G \setminus c(G)$ are trees, either free-floating or attached at vertices of $c(G)$.  The edges of $G$ are either in one of these trees or in $c(G)$.  Thus the boundary operator on the hairy graph complex $\hairy_V$ is the sum of two  operators $\bdry_T + \bdry_C$, which contract the (internal) tree-edges and the core edges of a hairy graph respectively.  

For many operads, including $\comm, \assoc$ and $\lie$, the part of hairy graph homology involving the boundary operator $\bdry_T$ is concentrated in one dimension, thus significantly simplifying the computation.  This simplification is accomplished by decomposing the hairy $k$-chains   as $$C_k\hairy=\bigoplus_{t+c=k}C^{t,c}(\hairy),$$
 where $C^{t,c}(\hairy)$ is generated by hairy graphs whose core has $c$ vertices, with the remaining $t$ internal vertices in the trees.  This gives a double complex with vertical boundary operator $\bdry_T$ and horizontal boundary operator $\bdry_C$, whose associated total complex is 
 $\hairy$.   For $\comm, \assoc$ and $\lie$, the homology of any column of this double complex vanishes except in  the bottom dimension, so the spectral sequence associated to the vertical filtration collapses to a single row, i.e. we obtain a new chain complex $\overline{\hairy}$ for computing $H_*(\hairy)$ with terms $\overline{\hairy}_k=H_0(C^{*,k}(\hairy))$ and boundary maps induced by $\bdry_C: C^{0,k}(\hairy)\to C^{0,k-1}(\hairy)$.   
  
As we will see in the next section, it is  useful to describe the above process in a different way, by forming  a new operad $\BO$ based on rooted trees, and considering components of a hairy $\cO$ graph which are not based on trees as (non-hairy) $\BO$ graphs based on their cores. For $\comm, \assoc$ and $\lie$  the homology of $\BO$ is zero except in dimension 0 and $H_0(\BO)$ itself has the structure of a cyclic operad.  If we now replace  $\BO$ graphs by  $H_0(\BO)$-graphs, the chain complex computing graph homology is nearly the same as the one we found using the double complex (not exactly the same because we are ignoring the tree components; also, we must allow bivalent vertices in this graph homology, since core graphs may have bivalent vertices).    Here are the details of this approach.  

We first define the operad $\BO$. Each generator  is based on a tree with a distinguished internal vertex.  All internal vertices of the tree are at least trivalent and are $\cO$-colored, the distinguished vertex has $\ell\geq 2$ leaves attached, numbered $0$ to $\ell-1$, and all other leaves  are labeled by elements of $V$.   
Here is a simple example, where the distinguished vertex is colored by $o_1$:  
$$
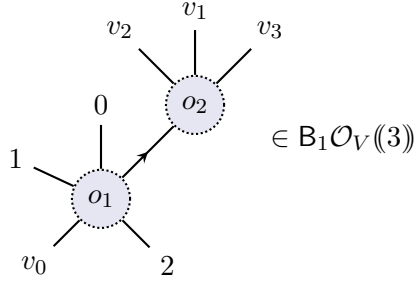
\begin{minipage}{3.5cm}
\begin{tikzpicture}[thick, scale=.5]
\tikzstyle{operad}=[draw,shape=circle, densely dotted, fill=blue!10];
\node[operad]  (o1) at (0,0) {$o_1$};
\node[operad]  (o2) at (2.5,2.5) {$o_2$};
\node (0) at (90:2.5) {$0$};
\node (1) at (155:2.5) {$1$};
\node (2) at (-45:2.5) {$2$}; 
\node (v0) at (-135:2.5) {$v_0$};
\node (v1) at (2.5,5) {$v_1$};
\node (v2) at (.5,2.5+2) {$v_2$};
\node (v3) at (4.5,2.5+2) {$v_3$};
\begin{scope}[decoration={markings,mark = at position 0.5 with {\arrow{stealth}}}]
\draw[postaction={decorate}] (o1) -- (o2);
\end{scope}
\draw (o1) -- (0)
(o1) -- (1)
(o1) -- (2)
(o1) -- (v0)
(o2) -- (v1)
(o2) -- (v2)
(o2) -- (v3);
\end{tikzpicture}
\end{minipage}
\in \CB_1\cO_V\pp{3}
$$
 The symmetric group $\Sym_n$ acts on   $\BO\pp{n}$ by permuting the labels $\{0,\ldots,n-1\}$.  Operad composition wedges the trees together at their distinguished vertices and uses operad composition in $\cO$  to combine the elements of $\cO$ at those vertices via the i/o slots at the numbered leaves.

Thus a hairy $\cO$-graph based on a graph $X$ with no tree components is equivalent to a (non-hairy!) $\BO$-graph based on the core $c(X)$ (See Figure~\ref{fig:OtoBO}).
\begin{figure}
 \begin{tikzpicture}
\tikzstyle{operad}=[draw,shape=circle, densely dotted, fill=blue!10];
\node[operad]  (o1) at (-1.5,0) {$o_1$};
\draw [dotted] (-2.1,0) ellipse (1.2 cm and .8cm);
\node[operad] (o2) at (1.5,0) {$o_2$};
\node[operad] (o3) at (0,1.2) {$o_3$};
\node[operad] (o4) at (0,2.5) {$o_4$};
\draw [dotted] (0,2.4)  ellipse  (1.1cm and 1.8cm);
\node[operad] (o5) at (3,0) {$o_5$};
\draw [dotted] (2.75,0) ellipse (2cm and 1cm);
\node (v1) at (-.7,3.2) {$v_1$};
\node (v2) at (.7,3.2) {$v_2$};
\node (v3) at (-2.5,.5) {$v_3$};
\node (v4) at (-2.5,-.5) {$v_4$};
\node (v5) at (4,.5) {$v_5$};
\node (v6) at (4,-.5) {$v_6$};
\node (v7) at (.7,1.2+.7) {$v_7$};
 \begin{scope}[decoration={markings,mark = at position 0.5 with {\arrow{stealth}}}]
  (o1) to [out=70, in=190]   (o3);
   \draw[postaction={decorate}]     (o1) to [out=70, in=190]   (o3);
   \draw[postaction={decorate}]  (o3) to [out=-10, in=110]   (o2);
   \draw[postaction={decorate}] (o1) -- (o2);
   \draw[postaction={decorate}] (o1) to [out=-70, in=-110]   (o2);
   \draw[postaction={decorate}]   (o2) -- (o5);
    \draw[postaction={decorate}]   (o3) -- (o4);
 \end{scope} 
\draw 
(o1) -- (v3)
(o1) -- (v4)
(o3) -- (v7)
(o4) -- (v2)
(o5) -- (v5)
(o5) -- (v6)
(o4) -- (v1);
\end{tikzpicture}
\hbox{\hskip 1in}
 \begin{tikzpicture}
\tikzstyle{every node}=[draw,shape=circle, densely dotted, fill=blue!30];
\node   (o1) at (-1.5,0) {$Bo_1$};
\node (o2) at ( 1.5,0) {$Bo_2$};
\node (o3) at (0,1.2) {$Bo_3$};
 \begin{scope}[decoration={markings,mark = at position 0.5 with {\arrow{stealth}}}]
  (o1) to [out=70, in=190]   (o3);
   \draw[postaction={decorate}]     (o1) to [out=70, in=190]   (o3);
   \draw[postaction={decorate}]  (o3) to [out=-10, in=110]   (o2);
   \draw[postaction={decorate}] (o1) -- (o2);
   \draw[postaction={decorate}] (o1) to [out=-70, in=-110]   (o2);
 \end{scope}
\end{tikzpicture} 
\caption{Haircut:  A hairy $\cO$-graph becomes a non-hairy $\BO$-graph}\label{fig:OtoBO}
\end{figure}
\medskip

The operad $\BO$ is a differential graded operad. It has the usual graph homology differential which contracts internal edges   in the underlying tree while composing the operad elements at either end of the edge.  
The $k$-chains are generated by trees with $k+1$ vertices, i.e. $k$ vertices which are not the distinguished vertex.  

We next show that for our usual three operads,  the homology $H_k(\BO)$ is  zero unless $k=0$.  In this case $H_0(\BO)$ inherits  the structure of a cyclic operad from the operad structure on  $\BO$.

 \subsection{$\cO=\comm$}  Recall that  the commutative operad is generated by star graphs, i.e. graphs with one internal vertex and at least three edges, which serve as i/o slots. Thus coloring the vertices of a tree by operad elements adds no extra structure to generators of $\BComm$.
 
\begin{theorem} \label{thm:commutative}   
$H_0(\BComm)=\bigoplus_{\ell\geq 2} H_0(\BComm)\pp{\ell},$ where $H_0(\BComm)\pp{\ell}$ is spanned by commutative spiders with $\ell$ i/o slots and at most one $V$-labeled leg.     $H_k(\BComm)=0$ for all $k>0$.  
\end{theorem}

\begin{proof} We first compute $H_0(\BComm)$.  
The $0$-chains $C_0(\BComm)$ are generated by star graphs, some of whose leaves are labeled by $\{0,\ldots,\ell-1\}$ for some $\ell\geq 2$ and the rest by elements of $V$. If $\ell=2$ then there must be at least one $V$-labeled leaf.    Any generator with two or more   $V$-labeled leaves is in the image of the boundary map.  There are no further relations, so generators with at most one $V$-labeled leaf form a basis for $H_0(\BComm)$  (see Figure~\ref{fig:comm}) and $H_0(\BComm)$ decomposes as a direct sum according to the number $\ell$ of numbered leaves, as claimed.\begin{figure}
 \begin{tikzpicture}[scale=.6]
\node (0) at (-2,0) {$0$};
\node (1) at (-1,0) {$1$}; 
\node (2) at (0,0) {$2$}; 
\node (dots) at (1,0) {$\ldots$};
\node (n-1) at (3,0) {$\ell-1$}; 
  \draw
 (.5,2) to (0)
 (.5,2) to (1)
  (.5,2) to (2)
   (.5,2) to (n-1);
\end{tikzpicture}
\hbox{\hskip 1in}
\begin{tikzpicture}[scale=.6]
\node (0) at (-2,0) {$0$};
\node (1) at (-1,0) {$1$}; 
\node (2) at (0,0) {$2$}; 
\node (dots) at (1,0) {$\ldots$};
\node (n-1) at (3,0) {$\ell-1$}; 
\node (v) at (.5,3) {$v$};
  \draw
 (.5,2) to (0)
 (.5,2) to (1)
  (.5,2) to (2)
   (.5,2) to (n-1)
   (.5,2) to (v);
\end{tikzpicture}
\caption{Generators of $H_0(\BComm)\pp{\ell}$}\label{fig:comm}
\end{figure}
 
To compute the higher homology groups, consider a generator $T$ of $C_k(\BComm)$ for $k\geq 1$.  Collapsing all of the internal edges of $T$ yields a  star graph $T_0$ with edges labeled $0,\ldots,\ell-1$  and $v_1,\ldots,v_s$; note that  $s\geq 1.$   Consider the subcomplex $C(T_0)$ of $C_*(\BComm)$ spanned by trees which collapse to $T_0$.  This is the augmented chain complex of the geometric realization of the partially ordered set of such trees, where the poset relation is forest collapse. (The augmentation is provided by the span of $T_0$).  Let  $T_1\in C_1(\BComm)$ be the tree with one internal edge $e$ separating the numbered leaves from the $V$-labeled leaves.  Then $e$ can be added to any other tree in the poset, i.e. the maps $T\to T\cup e\to T_1$ are poset maps contracting the poset to the single point $T_1$.  Thus the geometric realization of the poset is contractible, so its augmented chain complex   $C(T_0)$ is acyclic.  Hence $C(T_0)$ contributes nothing to $H_*(\BComm)$. Since every generator of $C_k(\BComm)$ for $k\geq 1$ is in some $C(T_0)$,   this shows  that $H_k(\BComm)=0$ for all $k\geq 1$.  
\end{proof}

The operad structure on  $H_0(\BComm)\pp{\ell}$  induced by the operad structure on $\BComm$ has the property that the composition of any two generators which each have a $V$-labeled leg is zero.  If $\ell=2$, this means that the composition of any two generators is zero, since all generators have a $V$-labeled leg.

 \subsection{$\cO=\assoc$}  Recall that the generators of the associative operad are {\em planar} star graphs, so the generators of $\BComm$ are endowed with a cyclic ordering at each internal vertex.  
 
 \begin{theorem} \label{thm:associative} $H_0(\BComm)\iso \bigoplus_{\ell\geq 2} H_0(\BComm)\pp{\ell}$, where $H_0(\BComm)\pp{\ell}$
 is spanned by associative spiders with $\ell$ i/o slots, with at most one $V$-labeled leg between adjacent slots. For all $k>1$,
$H_k(\BComm)=0$.  
\end{theorem}

\begin{proof} We first compute $H_0(\BComm)$.  The $0$-chains are generated by planar stars, some of whose leaves are labeled by $\{0,\ldots,\ell-1\}$ for some $\ell\geq 2$ and the rest by elements of $V$.  Any generator with two or more  {\em adjacent} $V$-labeled leaves is in the image of the boundary map, and there are no further relations, so these form a basis for $H_0(\BComm).$   We can express this algebraically by artificially adding   $1\in \F$ between i/o slots if there is no $V$-labeled leaf there, giving the formula in the statement of the theorem.
\begin{figure}
  \begin{tikzpicture}[scale=.6]
\node [draw,shape=circle] (origin) at (0,0) {};
\node (0) at (0:2) {$2$};
\node (1) at (60:2) {$0$}; 
\node (2) at (120:2) {$v_1$}; 
\node (3) at (180:2) {$1$};
\node (4) at (240:2) {$3$};
\node (5) at (300:2) {$v_2$}; 
  \draw
 (0,0) to (0)
 (0,0) to (1)
  (0,0) to (2)
(0,0) to (3)
(0,0) to (4)
(0,0) to (5);
\end{tikzpicture}
\hbox{\hskip .5in}
 \begin{tikzpicture}[scale=.6]
\node [draw,shape=circle] (origin) at (0,0) {};
\node (0) at (0:2) {$0$};
\node (1) at (90:2) {$1$}; 
\node (2) at (180:2) {$2$};
\node (3) at (270:2) {$3$};
  \draw
 (0,0) to (0)
 (0,0) to (1)
  (0,0) to (2)
   (0,0) to (3);
\end{tikzpicture}
\caption{Typical  generators of  $H_0(\BComm)\pp{4}$}\label{fig:assoc}
\end{figure}

For $k\geq 1$, the homology groups $H_k(\BComm)$ vanish by an argument similar to the $\cO=\comm$ case.  Here the role of $T_1$ is played by the planar tree in which every adjacent group of at least two $V$-labeled leaves is  separated from the base vertex by a single edge (see Figure~\ref{Tone}).

\begin{figure}
 \begin{tikzpicture}[scale=.6]
\node (0) [above right] at (45:2){$0$};
\node (1) [below right] at (-45:2) {$1$};
\node (2) [below] at (-90:2) {$2$};
\node (3) [above left] at (135:2) {$3$};
\node (4) [left] at (180:2) {$4$};
\node (v0) [above] at (0,2) {$v_0$};
\node (v1) [above right] at (3,1) {$v_1$};
\node (v2) [right] at (3,0) {$v_2$};
\node (v3) [below right] at (3,-1) {$v_3$};
 \node (v4) [left] at (-2,-1) {$v_4$};
 \node (v5) [below] at (-1,-2) {$v_5$};
\draw (4) -- (v2);
\draw (2) -- (v0);
\draw (1) -- (3);
\draw (0) -- (-1,-1)-- (v4);
\draw (v1) -- (2,0) -- (v3);
\draw (v5) -- (-1,-1);
\draw (0,0) circle (.3);
\draw (-1,-1) circle (.25);
\draw (2,0) circle (.25);
\end{tikzpicture}
\caption{Figure for proof of Theorem~\ref{thm:associative}}\label{Tone}
\end{figure}
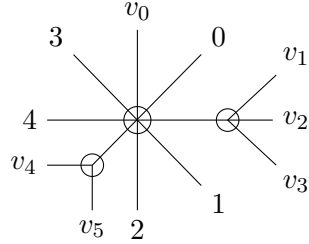 

\end{proof}
 
 The operad composition laws on $H_0(\BComm)$  are induced from the operad composition in $\BComm$. If two $V$-labeled hairs become adjacent after composition, the result is zero.

 \subsection{$\cO =\lie$}  Recall that the generators of $\cO$ are planar trivalent trees, modulo IHX (Jacobi) and AS (antisymmetry) relations.  These relations make this case more complicated than the previous two, and we break it up into smaller pieces. 
 
 \begin{proposition} \label{thm:HoneLie}
$H_0(\BLie)=\bigoplus_{\ell\geq 2} H_0(\BLie)\pp{\ell}$, where $H_0(\BLie)\pp{\ell}$ is spanned by Lie spiders with $\ell$ legs, each of which may have  commuting $V$-labeled hairs attached on the positive side, modulo the relation that the sum of spiders obtained by adding a hair  labeled $v$ to each leg is zero.  
\end{proposition}

\begin{proof} The $0$-chains are generated by trivalent planar trees,  with  $\ell$ leaves labeled $\{0,\ldots,\ell-1\}$ for some $\ell$ and the rest of the leaves labeled by elements of $V$. If   two edges coming from a common vertex end in $V$-labeled leaves, then the tree is in the image of the boundary map and we may disregard it.  Thus each generator has the form of a planar trivalent tree $T_0$ on $\{0,\ldots,\ell-1\}$ with single-edge ``hairs" attached along the edges, labeled by elements of $V$:
 
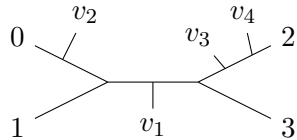
\begin{center} \begin{tikzpicture}[scale=.6]
\node (0) at (-2,1) {$0$};
\node (1) at (-2,-1) {$1$};
 \node (2) at (4,1) {$2$};
 \node (3) at (4,-1) {$3$};  
\draw (0) -- (0,0) -- (2,0) -- (2);
\draw (3) -- (2,0);
\draw (0,0) -- (1);
\node (v1) at (1,-1) {$v_1$};
\draw (1,0) to (v1);
\node (v2) at (-.5,1.5) {$v_2$};
\draw (-1,.5) to (v2);
\node (v3) at (2,1) {$v_3$};
\node (v4) at (3,1.5) {$v_4$};
\draw (2.6,.3) to (v3);
\draw (3.2,.6) to (v4);
\end{tikzpicture}
\end{center}
 
Such trees do not form a basis for the $0$-chains, but must be considered modulo IHX and AS relations.   
Using IHX relations we may express each generator as a sum of trees with no hairs on the internal edges.    Using AS relations we may flip all the hairs to the same (positive) side of their edges.  Another application of the IHX relation shows that the order of adjacent hairs may be interchanged (see Figure~\ref{fig:HairsCommute}; the second equality is due to the IHX relation).  Thus $H_0(\BLie)\pp{\ell}$ is generated by planar trivalent hairy $\lie$ trees $T_0$, leaves labeled by $\{0,\ldots,\ell-1\}$ with a monomial in $S(V)$ associated to each leaf.  IHX relations force the final relation among the generators.  This is illustrated in Figure~\ref{fig:Xtra}  for a $T_0$ with $4$ leaves, where two applications of IHX are required to move the hair labeled $v$ around, and the AS relation is used at times to flip the hair to the side indicated.

\begin{figure}
 \begin{tikzpicture}[scale=.4]
\node (0) at (-2,2) {$0$};
\node (1) at (-2,-2) {$1$};
 \node (2) at (3,2) {$2$};
 \node (3) at (3,-2) {$3$};  
\draw (0) -- (0,0) -- (1,0) -- (2);
\draw (3) -- (1,0);
\draw (0,0) -- (1);
\node (v) at (.5,2) {$v$};
\draw (-.5,.5) to (v);
\node (plus) at (4,0) {$+$};
\end{tikzpicture} 
 \begin{tikzpicture}[scale=.4]
 \node (0) at (-2,2) {$0$};
\node (1) at (-2,-2) {$1$};
 \node (2) at (3,2) {$2$};
 \node (3) at (3,-2) {$3$};  
\draw (0) -- (0,0) -- (1,0) -- (2);
\draw (3) -- (1,0);
\draw (0,0) -- (1);
\node (v) at (3,-.5) {$v$};
\draw (1.5,.5) to (v);
\node (plus) at (4,0) {$+$};
\end{tikzpicture}
\begin{tikzpicture}[scale=.4]
 \node (0) at (-2,2) {$0$};
\node (1) at (-2,-2) {$1$};
 \node (2) at (3,2) {$2$};
 \node (3) at (3,-2) {$3$};  
\draw (0) -- (0,0) -- (1,0) -- (2);
\draw (3) -- (1,0);
\draw (0,0) -- (1);
\node (v) at (.5,-2) {$v$};
\draw (1.5,-.5) to (v);
\node (plus) at (4,0) {$+$};
\end{tikzpicture}
 \begin{tikzpicture}[scale=.4]
  \node (0) at (-2,2) {$0$};
\node (1) at (-2,-2) {$1$};
 \node (2) at (3,2) {$2$};
 \node (3) at (3,-2) {$3$};  
\draw (0) -- (0,0) -- (1,0) -- (2);
\draw (3) -- (1,0);
\draw (0,0) -- (1);
\node (v) at (-2,.5) {$v$};
\draw (-.5,-.5) to (v);
\node (plus) at (4,0) {$=0$};
\end{tikzpicture}
\caption{Relation among generators of $H_0(\BLie)\pp{4}$}\label{fig:Xtra}
\end{figure}
\end{proof}

\begin{theorem} 
$H_k(\BLie)=0$ for $k>0$.
\end{theorem}
\begin{proof}The $k$-chains $C_k(\BLie)$ are spanned by labeled  $\lie$-trees, i.e. by  edge-oriented trees $T$ with $k+1$ $\lie$-decorated internal vertices, including a distinguished vertex with some of its leaves labeled $\{0,\ldots,\ell-1\}$ for some $\ell\geq 2$.  All other leaves of $T$ are labeled by elements of $V$.  The operad elements decorating the vertices are trivalent planar trees $\tau_1,\ldots,\tau_k$ modulo $AS$ and $IHX$ relations.  

Recall from  (\cite{CV}, Section 3.1) that the orientation data of any  $\lie$-graph can be simplified by replacing it with an ordering of the internal vertices of the forest $\Phi=\{\tau_1,\ldots,\tau_k\}$. A generator of $C_k(\BLie)$ then becomes a  {\em rooted forested tree}, i.e. a labeled trivalent tree $T$  together with an edge-ordered forest $\Phi,$ modulo IHX relations on the edges of $\Phi$.  The forest $\Phi$ must contain all the internal vertices of $T$, and  all of the leaves labeled by numbers must be attached to a single component of $\Phi$   (see Figure~\ref{forestedTree}).  
 \begin{figure}
 \begin{tikzpicture}[scale=.7]
\node (0) [left] at (-2,-1){$0$};
\node (1) [right] at (2,-1) {$1$};
\node (2) [below] at (-1,-2) {$2$};
\node (v1) [below] at (1,-2) {$v_1$};
\node (v2) [right] at (2,2) {$v_2$};
\node (v3) [above] at (1,3) {$v_3$};
 \node (v4) [above] at (-1,3) {$v_4$};
 \node (v5) [left] at (-2,2) {$v_5$};
 \node (e1) at (.3,1.8) {$e_1$};
  \node (e2) at (.8,-.2) {$e_2$};
   \node (e3) at (-.8,-.2) {$e_3$};
\draw (v5) -- (-1,2) -- (0,1) -- (1,2) -- (v2);
\draw (0) -- (-1,-1) -- (0,0) -- (1,-1) -- (1);
\draw (0,0) -- (0,1);
\draw (-1,2) -- (v4);
\draw (1,2) -- (v3);
\draw (-1,-1) -- (2);
\draw (1,-1) -- (v1);
\draw [line width=.7mm] (0,1) -- (1,2);
\draw [line width=.7mm] (-1,-1) -- (0,0) -- (1,-1);
\draw [fill=black] (-1,2) circle (.10);
\end{tikzpicture}
 \caption{A rooted forested tree}\label{forestedTree}
 \end{figure}
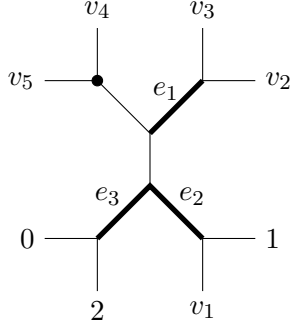
 The edge-orientation of the original $\lie$-tree, its planar embedding and the AS relation have been subsumed by the edge-ordering of $\Phi$.  The boundary map becomes $\bdry(T,\Phi)=\sum_e (T, \Phi\cup e)$, where the sum runs over all internal edges of $T$.  Note that adding an edge to $\Phi$ reduces the number of forest components by $1$ since $\Phi$ includes all internal vertices of $T$.  

Since the boundary map has no effect on the labels of $T$   we may compute the homology separately for each set of boundary labels.  If we restrict to $V$-labels which are basis elements, then the homology is a direct sum over all of the possible label sets. 
 Fix such a label set $L=\{0,\ldots,\ell-1,v_1,\ldots v_s\}$ with $|L|\geq 3$ elements, where $\ell\geq 2$ and $s\geq 0$.  A $\lie$ tree then has $d=|L|-2$ internal vertices. We may assume that all labels are distinct. Otherwise we are considering the complex of coinvariants under the automorphism group of the label set, and if the homology of the complex with distinct labels vanishes, so does the homology of the complex of coinvariants under any finite group action, since we are taking homology with trivial coefficients in a field of characteristic $0$.

The IHX relation can be encoded using the observation that there are three ways of expanding a 4-valent vertex into an edge and two trivalent vertices.  The fact that we must consider trees modulo  IHX relations thus leads us to consider more general rooted, forested trees {\em without} IHX equivalence,  which we arrange into a double complex as follows.

Define $C_{p,q}$ to be the vector space spanned by forested trees $(T,\Phi)$ such that $\Phi=\{\tau_1,\ldots,\tau_p\}$ has $p$ subtrees (including isolated vertices) and $T$ has $d-q+1$ internal vertices.  The $C_{p,q}$ form a double complex whose horizontal boundary operator $\bdry_h\colon C_{p,q}\to C_{p-1,q}$ sums over adding an edge to $\Phi$ in all possible ways, and whose vertical  boundary operator $\bdry_v\colon C_{p,q}\to C_{p,q-1}$ sums over expanding a vertex into an edge in all possible ways.  In both cases the new forest edge is put last in the ordering.

$$
\xymatrix{
C_{1,d}\ar[d]_{\bdry_v}\\
C_{1,d-1}\ar[d] &C_{2,d-1} \ar[l]_{\bdry_h}\ar[d]\\
C_{1,d-2}\ar[d]&C_{2,d-2}\ar[d] \ar[l]&C_{3,d-2}\ar[l]\ar[d]   \\
\vdots\ar[d]&\vdots\ar[d] & \vdots\ar[d]\\
C_{1,1}\ar[d]&C_{2,1} \ar[l]\ar[d]&C_{3,1} \ar[l]\ar[d]&\ldots\ar[l]&C_{d-\ell+1,1}\ar[l]\ar[d] \\
C_{1,0}&C_{2,0} \ar[l]&C_{3,0} \ar[l]&\ldots\ar[l]&C_{d-\ell+1,0}\ar[l]&C_{d-\ell+2,0} \ar[l]\\
}
$$
Thus $C_{1,d}$ is one-dimensional, spanned by $(S,v)$ where $S$ is the the star graph  and $v$ is the  forest consisting of its single internal vertex.  The bottom row is spanned by pairs $(T,\Phi)$ with $T$ trivalent, and the second row by pairs $(T,\Phi)$ such that $T$ has precisely one 4-valent vertex.  The left-hand column is spanned by pairs $(T,T^\prime)$ where $T^\prime$ consists of all internal edges of $T$ (i.e. $T'$ is obtained by pruning all the leaves from $T$).  

We consider the two spectral sequences associated to the vertical and horizontal filtrations of this double complex.  For the vertical filtration, the bottom row of the $E_1$-page is forested graphs modulo $IHX$, i.e. it is precisely the chain complex $C_*(\BLie)$ with the Lie boundary operator $\bdry_\lie$.   Furthermore, we claim that all other rows of the $E_1$ page are zero, so this spectral sequences converges to   $H_*(\BLie)$.    

Note that the top term of each column is spanned by rooted forested trees where each forest is a single vertex.  The terms lower in the column are formed by blowing up these vertices (separately) into trees.  Thus the columns are the augmented cochain complexes of  simplicial joins of (geometric realization of) posets of trees on a fixed set of leaves.  These posets are spherical (\cite{V}, Theorem 2.4) so have (co)homology only in the top dimension, corresponding to trivalent trees.  With our notation the trivalent trees are at the bottom of the column, which means that the homology of each column vanishes except in the lowest dimension. 

Next we consider the horizontal filtration.  Each row breaks up into a sum of subcomplexes $C_T$ according to the tree $T$ in the pair $(T,\Phi)$.    Consider the subtree $T_0$ of $T$ spanned by all the numbered leaves and its pruned  subtree $T_0^\prime=T_0\cap T'$.  If $T_0^\prime=T^\prime$  then this is the only forest  which makes $(T,T^\prime)$ into a legitimate rooted forested graph, since removing any edge from $T^\prime$  would violate the requirement that all numbered leaves be attached to the same component of $\Phi$.   Thus if $T_0^\prime=T^\prime$  then $C_T$ is one-dimensional, spanned by $(T,T'),$ which is in the left-hand column.     If $T_0^\prime$ is {\em not} equal to $T^\prime$,  then adding any set of edges of $T^\prime-T_0^\prime$ to $T_0^\prime$ gives a rooted forested tree, so  $C_T$ is the augmented chain complex of the simplex spanned by the edges of $T^\prime-T^\prime_0;$ in particular it is acyclic.  Thus the $E^1$ page of the spectral sequence is zero except for the left-hand column, which is spanned by pairs $(T,T^\prime)$ with $T^\prime=T_0^\prime$, and the vertical boundary operator blows up vertices into edges in all ways which preserve the condition $T^\prime=T_0^\prime$.  The left-hand column on the $E^1$-page is therefore  the augmented cochain complex of the poset $\mathfrak T_L$ of trees labeled by $L$ which satisfy $T^\prime=T_0^\prime$, and it suffices to show that the geometric realization of $\mathfrak T_L$ is spherical of dimension $d$.

 To prove this we use the standard translation of trees into compatible sets of partitions of the leaves (see, e.g. \cite{V}).   Each partition corresponds to an internal edge of the tree, and each side of a partition must have at least two elements.   The condition $T^\prime=T_0^\prime$ is equivalent to the condition that every edge of $T^\prime$ must separate two numbered leaves from each other, i.e. each partition has numbered leaves on both sides; call such a partition {\it OK}.  The complex $\mathfrak T_L$ is the barycentric subdivision of the complex of OK partitions, which by abuse of notation we also denote by $\mathfrak T_L$.

 We prove the theorem by  induction on the number of $V$-labels in $L$.  If there are no $V$-labeled leaves, then all partitions are OK and we are simply considering the poset of trees on $\ell$ leaves, which we know is spherical of dimension $d$.  If $L$ contains  a $V$-label $a$, let $A$ be the partition with one side equal to $\{0,a\}$; this partition is OK since it separates $0$ and $1$.   The subcomplex of $\mathfrak T_L$  spanned by partitions compatible with $A$ is a cone $c_A$ with vertex $A$, hence contractible.   We build the entire realization by adding the remaining partitions one at a time to this cone.  We add them in order of increasing {\em size}, where we measure the size by counting the number of leaves on the same side as $0$ (and hence the opposite side from $a,$ since they are not compatible with $A$). 
 
 We will abuse notation by writing  $X$ to mean the partition $\{X,L-X\}$, so $A= \{0,a\}.$ 
  If $B$ is not compatible with $A$ and has size $2$, then $B=\{0,x\}$ for some $x\neq a$.  Every partition in $lk(B)\cap c_A$  is compatible with $\{0,a,x\}=B\cup A,$ i.e. $lk(B)\cap c_A$ is a cone and is hence contractible.  After we have added all $B$ of size $2$ we continue with $B$ of size $3$, etc.  In each case the link with what has been previously added is a cone with vertex $B\cup A$ (note  that $B\cup A$ is an OK partition if $B$ is).   This continues until we reach size $d$, in which case $B\cup A$ has $d+1$ elements, so its complement is a singleton $\{y\}$ and $B\cup A$ is not a legitimate partition.   At this stage the intersection of $lk(B)$ with the complex we are building is the entire link, and we can identify this link with the complex of OK partitions on $L\setminus \{a\}$ by viewing $\{a,y\}$ as a single point.  This link is spherical of dimension $d-1$ by induction, so adding $B$ produces a complex homotopy equivalent to the suspension of this link, i.e. a wedge of spheres of dimension $d$.  Adding further partitions of size $d$ produces more suspended wedges of spheres, so the entire complex $\mathfrak T_L$ is homotopy equivalent to a  wedge of spheres of dimension $d$.

\end{proof}

 \begin{figure}
 $$0=\partial\left( 
\begin{minipage}{4.5cm}\begin{tikzpicture}[scale=.4]
\node [above left]   (0) at (-4,3)  {$0$};
\node [below left]   (1) at (-4,-3) {$1$};
\node [above right] (2) at (4,3)   {$2$};
\node [below right] (3) at (4,-3)  {$3$};  
\draw (0) -- (-1,0) -- (1,0) -- (2);
\draw (4,-3) -- (1,0);
\draw (-1,0) -- (1);
\draw (-2.5,1.5) -- (-2,2);
\draw [densely dotted] (-2,2) -- (-1.5, 2.5);
\node [above] (v1) at (-1,4) {$v_1$};
\draw (-1.5,2.5) -- (-1,3) -- (v1);
\node [right] (v2) at (0,3) {$v_2$};
\draw (-1,3)--(v2);
\node [below right] (v3) at (3.5,.5) {$v_3$};
\draw (2.5,1.5) -- (v3);
\node [above left] (v7) at (-3,0) {$v_7$};
\draw (-2,-1) to (v7);
\node [above left] (v8) at (-4,-1) {$v_8$};
\draw (-3,-2) to (v8);
\node [below left] (v4) at (.5,-1.5) {$v_4$};
\draw (1.5,-.5) to (v4);
\node [below left] (v5) at ( 1.5,-2.5) {$v_5$};
\draw (2.5,-1.5) to (v5);
\node [below left] (v6) at (2.5, -3.5)   {$v_6$};
\draw (3.5,-2.5) to (v6);
\end{tikzpicture} 
\end{minipage}\right)
=
\begin{minipage}{4.7cm}
 \begin{tikzpicture}[scale=.4]
\node [above left] (0) at (-4,3) {$0$};
\node [below left] (1) at (-4,-3) {$1$};
 \node [above right] (2) at (4,3) {$2$};
 \node  [below right] (3) at (4,-3) {$3$};  
\draw (0) -- (-1,0) -- (1,0) -- (2);
\draw (4,-3) -- (1,0);
\draw (-1,0) -- (1);
\node [above right] (v1) at (-2,3) {$v_1$};
\draw (-3,2) to  (v1);
\node [above right] (v2) at (-1,2) {$v_2$};
\draw (-2,1)--(v2);
\node [below right] (v3) at (3.5,.5) {$v_3$};
\draw (2.5,1.5) -- (v3);
\node [above left] (v7) at (-3,0) {$v_7$};
\draw (-2,-1) to (v7);
\node [above left] (v8) at (-4,-1) {$v_8$};
\draw (-3,-2) to (v8);
\node [below left] (v4) at (.5,-1.5) {$v_4$};
\draw (1.5,-.5) to (v4);
\node [below left] (v5) at ( 1.5,-2.5) {$v_5$};
\draw (2.5,-1.5) to (v5);
\node [below left] (v6) at (2.5, -3.5)   {$v_6$};
\draw (3.5,-2.5) to (v6);
\end{tikzpicture} 
\end{minipage}
-
\begin{minipage}{4.7cm}
 \begin{tikzpicture}[scale=.4]
\node [above left] (0) at (-4,3) {$0$};
\node [below left] (1) at (-4,-3) {$1$};
 \node [above right] (2) at (4,3) {$2$};
 \node  [below right] (3) at (4,-3) {$3$};  
\draw (0) -- (-1,0) -- (1,0) -- (2);
\draw (4,-3) -- (1,0);
\draw (-1,0) -- (1);
\node [above right] (v1) at (-2,3) {$v_2$};
\draw (-3,2) to  (v1);
\node [above right] (v2) at (-1,2) {$v_1$};
\draw (-2,1)--(v2);
\node [below right] (v3) at (3.5,.5) {$v_3$};
\draw (2.5,1.5) -- (v3);
\node [above left] (v7) at (-3,0) {$v_7$};
\draw (-2,-1) to (v7);
\node [above left] (v8) at (-4,-1) {$v_8$};
\draw (-3,-2) to (v8);
\node [below left] (v4) at (.5,-1.5) {$v_4$};
\draw (1.5,-.5) to (v4);
\node [below left] (v5) at ( 1.5,-2.5) {$v_5$};
\draw (2.5,-1.5) to (v5);
\node [below left] (v6) at (2.5, -3.5)   {$v_6$};
\draw (3.5,-2.5) to (v6);
\end{tikzpicture} 
\end{minipage}
$$
\caption{$V$-labeled hairs commute in $H_1(\BLie)$.}\label{fig:HairsCommute}
\end{figure}

As in the $\comm$ and $\assoc$ cases, $H_0(\BLie)$ inherits its operad structure from $\BLie$.

Recall that a connected hairy graph corresponds to a $\BO$-graph only if it is based on a graph with a non-empty core, i.e. on a graph which is not a tree.  The following definition  accounts for hairy graphs based on trees as well.  
\begin{definition}
Let $\cO$ be a cyclic operad such that  $H_k(\BO)$ is zero for $k>0$, and   
give   $H_0(\BO)$  the (cyclic) operad structure induced from $\BO$.  
The \emph{reduced hairy graph complex} $\overline{\hairy}$  is spanned by finite unions of finite connected graphs, where 
\begin{enumerate} 
\item the non-tree components are core graphs whose vertices are $H_0(\BO)$-colored, and
\item the internal vertices of the tree components are $\cO$-colored.
\end{enumerate}
\end{definition}
We note that in the commutative and associative cases, the only tree  components that survive in homology are tripods. In the Lie case, however, there are interesting homology classes based on trees. For example,
$$
\begin{minipage}{2.3cm}
\begin{tikzpicture}[thick]
\draw (0,0)--(-.866,.5) node[fill=white]{$v_1$};
\draw (0,0)--(.866,.5) node[fill=white]{$v_2$};
\draw (0,0)--(0,-.2);
\draw[densely dotted] (0,-.2)--(0,-.8);
\draw (0,-.8)--(0,-1);
\draw (0,-1)--(-.866,-1.5) node[fill=white]{$v_3$};
\draw (0,-1)--(.866,-1.5) node[fill=white]{$v_4$};
\end{tikzpicture}
\end{minipage}
-
\begin{minipage}{2.5cm}
\begin{tikzpicture}[thick, rotate=90]
\draw (0,0)--(-.866,.5) node[fill=white]{$v_3$};
\draw (0,0)--(.866,.5) node[fill=white]{$v_1$};
\draw (0,0)--(0,-.2);
\draw[densely dotted] (0,-.2)--(0,-.8);
\draw (0,-.8)--(0,-1);
\draw (0,-1)--(-.866,-1.5) node[fill=white]{$v_4$};
\draw (0,-1)--(.866,-1.5) node[fill=white]{$v_2$};
\end{tikzpicture}
\end{minipage}
+
\begin{minipage}{2.5cm}
\begin{tikzpicture}[thick, rotate=90]
\draw (0,0)--(-.866,.5) node[fill=white]{$v_3$};
\draw (0,0)--(.866,.5) node[fill=white]{$v_2$};
\draw (0,0)--(0,-.2);
\draw[densely dotted] (0,-.2)--(0,-.8);
\draw (0,-.8)--(0,-1);
\draw (0,-1)--(-.866,-1.5) node[fill=white]{$v_4$};
\draw (0,-1)--(.866,-1.5) node[fill=white]{$v_1$};
\end{tikzpicture}
\end{minipage}
\neq 0\in H_2(\hairy\lie_V)
$$
  
\begin{theorem}\label{thm:quasi}
Let $\cO$ be a cyclic operad for which $H_k(\BO)$ vanishes for $k>0$.  Then the hairy graph complex $\hairy$ is quasi-isomorphic to the reduced  hairy graph complex $\overline{\hairy}$.
\end{theorem}
\begin{proof}
The chain complex $$\ldots H_1(C_{*,c}(\hairy))\buildrel{\bdry_C}\over\to H_1(C_{*-1,c}(\hairy))\ldots $$  resulting from the double-complex decomposition of $\hairy$ is exactly the same as  the chain complex for the reduced hairy graph complex $\overline{\hairy}$.
\end{proof}


\section{The rank one subcomplex and  dihedral homology}\label{sec:IO}

In the last section we saw that for $\cO= \comm, \assoc$ and $\lie,$   $H_k(\BO)$ vanishes for $k\neq 0$; in particular, $H_k(\BO\pp{2})=0$ for $k\neq 0$.  This implies that we can give $H_0(\BO\pp{2})$ the structure of a (non-unital) associative algebra with involution.  In this section we show that the dihedral homology of this algebra, as defined by Loday in \cite{Loday}, can be identified with 
the homology of the subcomplex of $\hairy$ generated by connected graphs with cyclic fundamental group.   We then use known results on cyclic and dihedral homology to compute this homology for $\cO= \comm, \assoc$ and $\lie$, thus finding new non-trivial hairy graph homology classes for these operads.

\subsection{The associative algebras $\IOV$ and  $\AOV$}

Recall that   $\BO\pp2$ is generated by trees whose internal vertices (which are at least trivalent) are colored by operad generators,  and one vertex is distinguished.  There are two leaves  at the distinguished vertex labeled $0$ and $1$, and the rest of the leaves are labeled by elements of $V$ (see Figure~\ref{BO2}).  Operad composition defines an associative multiplication on these generators, using the leaves labeled $0$ as output slots and $1$ as input slots, turning   $\BO\pp2$ into a (non-unital) associative algebra.  \begin{figure}
$$
\begin{minipage}{3.5cm}
\begin{tikzpicture}[thick, scale=.5]
\tikzstyle{operad}=[draw,shape=circle, densely dotted, fill=blue!10];
\node[operad]  (o1) at (0,0) {$o_1$};
\node[operad]  (o2) at (2.5,2.5) {$o_2$};
\node (0) at (90:2.5) {$v_1$};
\node (1) at (155:2.5) {$v_2$};
\node (2) at (-45:2.5) {$1$}; 
\node (v0) at (-135:2.5) {$0$};
\node (v1) at (2.5,5) {$v_3$};
\node (v2) at (.5,2.5+2) {$v_4$};
\node (v3) at (4.5,2.5+2) {$v_5$};
\begin{scope}[decoration={markings,mark = at position 0.5 with {\arrow{stealth}}}]
\draw[postaction={decorate}] (o1) -- (o2); 
\end{scope}
\draw (o1) -- (0)
(o1) -- (1)
(o1) -- (2)
(o1) -- (v0)
(o2) -- (v1)
(o2) -- (v2)
(o2) -- (v3);
\end{tikzpicture}
\end{minipage}
$$
\caption{Element of $ \CB_2\cO_V\pp{2}$}\label{BO2}
\end{figure}
Since the higher graph homology groups of $\BO\pp2$ vanish, the multiplication on  the differential graded algebra $\BO\pp2$ induces a multiplication on its first homology.  We denote the resulting (non-unital) associative algebra by $\IOV$:
$$\IOV = H_0(\BO\pp2).$$
There is a standard way of turning a non-unital associative algebra $I$ into a unital one $I_+$, by setting $I_+=\F\oplus I$ with multiplication $(r,v)(s,w)=(rs,rw+sv+vw)$ (so the unit is $(1,0)$).    In the case of $\IOV$, we can interpret this construction as follows.  We add a generator to $\BO$  with only two leaves, labeled $0$ and $1,$ whose internal vertex is colored by the identity in $\cO$. This element acts as an identity for $\BO$ and also for the first homology $\IOV$, and we define  $\AOV:=(\IOV)_+$.  

Since $\cO$ is a cyclic operad, there is an involution which switches the output and input slots, thereby inducing an involution of   $\AOV$ that is an anti-automorphism of the algebra.

\subsection{Connection with dihedral homology} 

 For the convenience of the reader, we now   recall the definitions of cyclic and dihedral homology.  For more information and details, see  Loday's book \cite{Loday}.

Give an associative algebra $A$  the {\em Hochschild chain complex} $(C(A),b)$ is defined by $$C_k(A)=A^{\otimes k+1}\hbox{\qquad and \qquad}  b_k = \sum_{i=0}^k (-1)^i d_i\colon  A^{\otimes k+1}\to A^{\otimes k}$$ where
\begin{align*}
d_i(a_0\otimes\ldots\otimes a_k)&= a_0\otimes\ldots \otimes a_ia_{i+1}\otimes \ldots\otimes a_k \ \hbox{ if }\,  0\leq i <k\\ 
d_k(a_0\otimes\ldots\otimes a_k) &= a_ka_0\otimes\ldots\otimes a_{k-1}\end{align*}

If $A$ has an involution $a\mapsto \bar a,$ then dihedral homology is defined as the homology of the chain complex obtained by taking the quotient of $C(A)$ modulo natural actions of the dihedral groups $D_{2(k+1)}$ on $C_k(A)$  for each $k$.  Specifically, we use the presentation   
$$D_{2(k+1)}=\langle\tau_k,\omega_k\,|\, \tau_k^{k+1}=\omega_k^2, \omega_k\tau_k\omega_k^{-1}=\tau_k^{-1}\rangle$$
The generators $\tau_k$ and $\omega_k$ act on $C_k(A)$ by the formulas
\begin{align*}
\omega_k(a_0,\ldots,a_k)&=(\overline a_0,\overline a_k,\ldots,\overline a_1)\\
\tau_k(a_0,\ldots,a_k)&=({a_k},{a_0},\ldots,{a_{k-1}})
\end{align*}
The following definitions are justified by \cite{Loday}, Theorem 5.2.8.

\begin{definition}[Cyclic, Dihedral and Skew Dihedral homology in characteristic $0$] 
Let $A$ be an associative algebra over a field of characterisitic $0$, and define maps $C_k(A)\to C_k(A)$ for each $k$ by    $t_k=(-1)^k\tau_k$ and $y_k=(-1)^{k(k+1)/2}\omega_k$.  
 Then
\begin{enumerate}
\item The boundary map $b$ on $C(A)$ induces a boundary map on the quotient $C(A)/(1-t)$.  The homology of the resulting quotient complex is the {\em  cyclic homology} of $A$, denoted $\HC_*(A).$  
\item If $A$ has an involution $a\mapsto \bar{a}$, then the boundary map on $C(A)/(1-t)$ induces boundary maps on the quotients $C(A)/(1-t,1-y)$ and $C(A)/(1-t,1+y)$.  The homology of the quotient complex $C(A)/(1-t,1-y)$ is the {\em dihedral homology}  $\HD_*(A)$, and the homology of  $C(A)/(1-t,1+y)$ is the {\em skew dihedral homology} $\HD_*'(A).$ 
\end{enumerate}
\end{definition}

 The  relation between cyclic and  dihedral homology is very simple (see \cite{Loday}, 5.2.7.1), namely
there is a canonical isomorphism $$\HC_k(A)=\HD_k(A)\oplus \HD'_k(A).$$ 

The definition of cyclic and dihedral homology makes no reference to a unit in $A$, so also makes sense for non-unital algebras $I$.  If $I$ is  non-unital, then there are isomorphisms between the cyclic or dihedral homology of $I$ and the {\em reduced} cyclic or dihedral homology of $I_+$

We now make the connection between dihedral homology and the homology of the subcomplex $C_{*,1}\hairy$ of $\hairy$ spanned by connected hairy graphs $X$ of rank $1$, i.e. with cyclic fundamental group.  Recall from \cite{CV} that this is denoted $H_{*,1}(\hairy)$.  


\begin{proposition}\label{prop:dihedral} If $H_i(\BO\pp2)$ vanishes for $i>0$ then  the homology   of the rank one part of the hairy graph complex is isomorphic to the shifted dihedral homology of the   algebra $\IO=H_0(\BO\pp2)$:
$$H_{k,1}(\hairy)\iso \HD_{k-1}(\IOV)\iso {\overline{\HD}}_{k-1}(\AOV).$$
\end{proposition}

\begin{proof}  A rank one hairy graph has a core $c(X)$ consisting of a single circle with $k\geq 1$ vertices.  As we saw in Section~\ref{sec:BO} we can regard $X$ as $c(X)$ with vertices colored by elements of $\BO\pp2,$ where every edge of $c(X)$ is directed from the output slot of one operad element to the input slot of the next operad  element (see Figure~\ref{fig:rankone}).

\begin{figure}
 \begin{tikzpicture}
\tikzstyle{every node}=[draw,shape=circle, densely dotted, fill=blue!30];
\begin{scope}[decoration={markings,mark = at position 0.5 with {\arrow{stealth}}}]
\node (o1) at (90:1.5) {$Bo_1$};
\node (o2) at (18:1.5) {$Bo_2$};
\node (o3) at (-54:1.5) {$Bo_3$};
\node (o4) at (234:1.5) {$Bo_4$};
\node (o5) at (162:1.5) {$Bo_5$};
\midarrow (o1) to  (o2);
\midarrow (o2) to  (o3);
\midarrow (o3) to (o4);
\midarrow (o4) to    (o5);
\midarrow (o5) to (o1);
\end{scope}
\end{tikzpicture} 
\caption{Rank one $\BO$-graph}\label{fig:rankone}
\end{figure}

 The involution on $\BO\pp2$ changes the direction on each edge, and the boundary operator is the sum over edges of $c(X)$ of the result of multiplying adjacent algebra elements.   The involution on $\BO\pp{2}$ sends a generator  to itself or minus itself,  depending on the parity of $k$.  What we have just described is the chain complex for computing the dihedral homology of the non-unital algebra $\BO\pp{2}$, shifted by one.   
Theorem~\ref{thm:quasi} says that for $\cO=\comm, \assoc$ or $\lie$ this subcomplex is quasi-isomorphic to the graph homology complex (allowing bivalent vertices)  of circles with $k$ vertices colored by elements of $\IOV=H_0(\BO\pp{2})$, i.e. $$H_{k,1}(\hairy_V) = HD_{k-1}(\BO\pp{2})\iso \HD_{k-1}(\IOV)\iso \overline{\HD}_{k-1}(\AOV) .$$
\end{proof}

For $k=1$, this together with the definition of dihedral homology in dimension $0$ gives $H_{1,1}(\hairy_V)\iso \overline{\HD}_0(\AOV)\iso \AOV/([\AOV,\AOV] + \AO^\sigma_V),$ where $\AO^\sigma_V$ denotes the
subspace of elements in $\AOV$ which are fixed by the involution. 

If the homology of $\BO\pp2$ is not concentrated in dimension zero then, as we see in the proof of the proposition, we can still identify rank one hairy graph homology with a dihedral homology group, but we must use a graded version of dihedral homology and there is no reduction to the ordinary dihedral homology of an associative algebra.

\subsection{Calculations for $\cO=\lie$}

\begin{proposition} $A{\lie_V}\cong S(V)$ and the involution is multiplication by $-1$ on $V$.
\end{proposition}
\begin{proof} By Proposition~\ref{thm:HoneLie} $I\lie_V=H_0(\BLie\pp2)$ is generated by line segments whose ends are labeled $0$ and $1,$ with commuting $V$-labeled hairs attached on top:
 
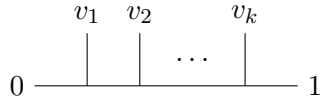
\begin{center}
\begin{tikzpicture}[scale=.7]
\node (0) [left] at (0,0) {$0$};
\node (1) [right] at (5,0) {$1$};
\node (v1) [above] at (1,1) {$v_1$};
\node (v2) [above] at (2,1) {$v_2$};
\node (vk) [above] at (4,1) {$v_k$};
\draw (0) -- (1);
\draw (1,0) -- (v1);
\draw (2,0) -- (v2);
\node (dots) at (3,.5) {$\ldots$};
 \draw (4,0) -- (vk);  
\end{tikzpicture}
  \end{center}
After adding a unit (with no hairs attached), we get  $A{\lie}_V\iso S(V)$ the symmetric algebra of $V$.  The involution sends $x\to -x$ if $x$ has an odd number of hairs attached, and fixes $x$ if the number of hairs is even.  This is because interchanging $0$ and $1$ makes it necessary to flip each hair to the other side, incurring a sign change for each flip by the $AS$ relation.
\end{proof}

The reduced cyclic homology of the symmetric algebra $S(V)$  is computed in  \cite{Loday}, Theorem 3.2.5, where it is given by the formula 
$$\overline{\HC}_n(S(V))= \Omega^n/d\Omega^{n-1}.$$
Here $\Omega^n=\Omega^n_{S(V)|\F}$ is the $S(V)$-module of differential $n$-forms on $V$, and $d\colon\Omega^{n-1}\to\Omega^n$ is the standard differential $d(adv_1\ldots dv_n)=dadv_1\ldots dv_n$.   

There is an isomorphism $S(V)\otimes \ext^n V\iso \Omega^n$ sending $a\otimes v_1\wedge\ldots \wedge v_n\mapsto adv_1\ldots dv_n$.  
Under this isomorphism the boundary map, restricted to homogeneous degree $k$ polynomials,  is   $d\colon S^{k+1}(V)\otimes \ext^{n-1} V\to S^{k}(V)\otimes \ext^n V$ or, in the notation of Schur functors
$$d\colon \SF{(k+1)}(V)\otimes \SF{(1^{n-1})}(V)\to \SF{(k)}(V)\otimes \SF{(1^n)}(V) $$
A simple case of the Littlewood-Richardson rule (see, e.g. \cite{FH}, Appendix A) shows
$\SF{(k)}(V)\otimes \SF{(1^n)}(V)=\SF{(k,1^n)}(V)\oplus \SF{(k+1,1^{n-1})}(V),$ so
$$d\colon \SF{(k+1,1^{n-1})}(V)\oplus \SF{(k+2,1^{n-2})}(V)\to \SF{(k,1^n)}(V)\oplus \SF{(k+1,1^{n-1})}(V).$$  Since $d$ is nonzero and is $GL(V)$-equivariant, it must surject onto the $\SF{(k+1,1^{n-1})}(V)$ component, which implies that  
$$\overline{\HC}_n(S(V))= \bigoplus_{k\geq 0} \SF{(k,1^{n})}(V).$$
The action of the involution $\sigma$ on $\SF{(k,1^{n})}(V)$ is  multiplication by $(-1)^{k+1}$. 
 We conclude that 
$$
H_{1,n+1}(\hairy\assoc_V)=\overline{\HD}_n(S(V))  = \bigoplus_{k\geq 0} \SF{(2k+1,1^{n})}(V). 
$$

When $n=0$, this recovers   the computation for $\cO=
\lie$ in \cite{CKV} that $H_{1,1}(\hairy_V)\cong \bigoplus_{k=0}^\infty S^{2k+1}(V).$

The method of differential forms used in this calculation fails for the commutative and associative operads because the corresponding algebras are not \emph{smooth} (see \cite{Loday}). Indeed in these cases the space of  $n$-forms $\Omega^n_{\AOV|\F}$ is trivial. 
\subsection{Calculations for $\cO=\comm$}

By Theorem~\ref{thm:commutative}, $I\comm_V=H_0(\BComm)\pp2$ is generated by tripods with leaves labeled by $0, 1$ and $v$ for some $v\in V$, i.e. $I\comm_V\iso V$.  The multiplication is trivial, since composing two such tripods results in a tree with two $V$-labeled leaves, which is zero in homology.   The involution is trivial since switching the labels $0$ and $1$ gives the same tripod. 

Since multiplication in $I\comm_V$ is trivial all of the differentials in the chain complex for ${\HD}(I\comm_V)$ are zero, so $$H_{k+1,1}(\hairy\comm_V)={\HD}_{k}(I\comm_V)=V^{\otimes k+1}/(1-t,1-y).$$ where  $t\cdot(v_0\otimes\cdots\otimes v_k)= (-1)^{k} v_1\otimes\cdots\otimes v_k\otimes v_0$ and $y\cdot (v_0\otimes\cdots\otimes v_k)=(-1)^{k+1+\binom{k+1}{2}}v_k\otimes\cdots\otimes v_0$.

Below we compute $H_{k,1}(\hairy_V)$ explicitly for small values of $k$:

$H_{1,1}(\hairy\comm_V)=HD_0(I\comm_V)=0$ since $y\cdot v_0=-v_0$.  

$H_{2,1}(\hairy\comm_V)=HD_1(I\comm_V)\iso \ext^2V,$ where $v_0\wedge v_1$ corresponds to the  hairy graph
$$
\begin{minipage}{3.5cm}
\begin{tikzpicture}[scale=.5]
\begin{scope}[decoration={markings,mark = at position 0.5 with {\arrow{stealth}}}] 
\node (v) [left] at (-2,0) {$v_0$};
\node (w) [right] at (2,0) {$v_1$};
\draw (1,0)--(w);
\draw (-1,0)--(v);
\midarrow (-1,0) to [out=90, in=90] (.5+.5,0);
\midarrow (1,0) to [out=-90, in=-90] (-1,0);
\end{scope}
\end{tikzpicture}
\end{minipage}
$$

$H_{3,1}(\hairy\comm_V)=HD_2(I\comm_V)\iso S^3V,$ where $v_0v_1v_2$ corresponds to the hairy graph
$$
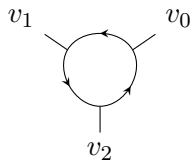
\begin{minipage}{3.5cm}
\begin{tikzpicture}[scale=.5]
\begin{scope}[decoration={markings,mark = at position 0.5 with {\arrow{stealth}}}] 
\node (v0)  [above right] at (30:1.7) {$v_0$};
\node (v1) [above left] at (150:1.7) {$v_1$};
\node (v2) [below] at (270:1.7) {$v_2$};
\draw (30:1)--(v0);
\draw (150:1)--(v1);
\draw (270:1)--(v2);
\midarrow (30:1) to [out=120, in=60] (150:1);
\midarrow (150:1) to [out=240, in=180] (270:1);
\midarrow (270:1) to [out=0, in=-60] (30:1);
\end{scope}
\end{tikzpicture}
\end{minipage}
$$
$H_{4,1}(\hairy\comm_V)=HD_3(I\comm_V)\iso \ext^4V \oplus \SF{(2,2)}(V).$ Here $v_0\wedge v_1\wedge v_2\wedge v_3$ corresponds to the following sum of hairy graphs
$$
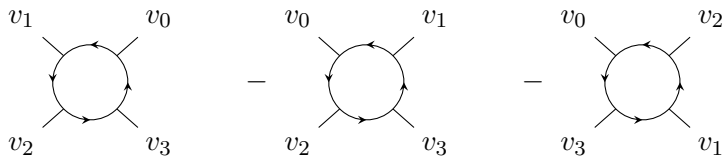
\begin{minipage}{3.2cm}
\begin{tikzpicture}[scale=.5]
\begin{scope}[decoration={markings,mark = at position 0.5 with {\arrow{stealth}}}] 
\node (v0) [above right] at (45:1.7) {$v_0$};
\node (v1) [above left] at (135:1.7) {$v_1$};
\node (v2) [below left] at (-135:1.7) {$v_2$};
\node (v3) [below right] at (-45:1.7) {$v_3$};
\draw (45:1)--(v0);
\draw (135:1)--(v1);
\draw (-135:1)--(v2);
\draw (-45:1)--(v3);
\midarrow (45:1) to [out=135, in=45] (135:1);
\midarrow (135:1) to [out=-135, in=135] (-135:1);
\midarrow (-135:1) to [out=-45, in=-135] (-45:1);
\midarrow (-45:1) to [out=45, in=-45] (45:1);
\end{scope}
\end{tikzpicture}
\end{minipage}
-
\begin{minipage}{3.2cm}
\begin{tikzpicture}[scale=.5]
\begin{scope}[decoration={markings,mark = at position 0.5 with {\arrow{stealth}}}] 
\node (v0) [above right] at (45:1.7) {$v_1$};
\node (v1) [above left] at (135:1.7) {$v_0$};
\node (v2) [below left] at (-135:1.7) {$v_2$};
\node (v3) [below right] at (-45:1.7) {$v_3$};
\draw (45:1)--(v0);
\draw (135:1)--(v1);
\draw (-135:1)--(v2);
\draw (-45:1)--(v3);
\midarrow (45:1) to [out=135, in=45] (135:1);
\midarrow (135:1) to [out=-135, in=135] (-135:1);
\midarrow (-135:1) to [out=-45, in=-135] (-45:1);
\midarrow (-45:1) to [out=45, in=-45] (45:1);
\end{scope}
\end{tikzpicture}
\end{minipage}
-
\begin{minipage}{3.2cm}
\begin{tikzpicture}[scale=.5]
\begin{scope}[decoration={markings,mark = at position 0.5 with {\arrow{stealth}}}] 
\node (v0) [above right] at (45:1.7) {$v_2$};
\node (v1) [above left] at (135:1.7) {$v_0$};
\node (v2) [below left] at (-135:1.7) {$v_3$};
\node (v3) [below right] at (-45:1.7) {$v_1$};
\draw (45:1)--(v0);
\draw (135:1)--(v1);
\draw (-135:1)--(v2);
\draw (-45:1)--(v3);
\midarrow (45:1) to [out=135, in=45] (135:1);
\midarrow (135:1) to [out=-135, in=135] (-135:1);
\midarrow (-135:1) to [out=-45, in=-135] (-45:1);
\midarrow (-45:1) to [out=45, in=-45] (45:1);
\end{scope}
\end{tikzpicture}
\end{minipage}
$$

and, setting  $\tau_{(ij)(kl)}$ to be the sum of permutations $\id+(ij)+(kl)+(ij)(kl)$, $\SF{(2,2)}(V)$ is generated by sums
$$
\tau_{(01)(23)}\directedwheel{v_0}{v_1}{v_2}{v_3}- \quad \tau_{(02)(13)}\directedwheel{v_0}{v_1}{v_2}{v_3}$$

A similar calculation can be done for $k=4$.  Using the abbreviation $n[\lambda]$ for the direct sum of $n$ copies of the Schur functor $\SF{\lambda}(V)$ for the partition $\lambda$, we have 
$$
\begin{array}{c|ccccc}
k&0&1&2&3&4\\
\hline
{HD}_k(I\comm_V)&0&[1^2]& [3] & [1^4]\oplus [2^2]&2[3,1^2]
\end{array}
$$

%

\subsection{Calculations for $\cO=\assoc$.} 
By Theorem~\ref{thm:associative}, $I\assoc_V=H_0(\BAssoc\pp2)$ is generated by the planar trees of the following form, where $u, u',v, v'\in V$;
\begin{center}
\begin{tikzpicture}[scale=.9]
 \node (I) at (-1.8,0) {$(a)$};
  \node (0a) at (-1,0) {$0$};
   \node (1a) at (1,0) {$1$};
    \node (va) at (0,1) {$v$};
    \draw (0a) -- (1a);
    \draw (0,0) -- (va);
 \draw  (0,0) circle (.2);
   \node (II) at (2.2,0) {$(b)$};
      \node (0b) at (3,0) {$0$};
   \node (1b) at (5,0) {$1$};
    \node (vb) at (4,-1) {$v'$};
    \draw (0b) -- (1b);
    \draw (4,0) -- (vb);
     \draw  (4,0) circle (.2);
    \node (III) at (6.2,0) {$(c)$};
      \node (0c) at (7,0) {$0$};
   \node (1c) at (9,0) {$1$};
    \node (vc) at (8,1) {$u$};
        \node (wc) at (8,-1) {$u'$};
    \draw (0c) -- (1c);
    \draw (wc) -- (vc);
     \draw  (8,0) circle (.2);
  \end{tikzpicture}
\end{center}
 The product of a generator of type (a) and one of type (b), in either order, is of type (c), and all other products are zero.   Thus $I\assoc_V\iso V \oplus V\oplus (V\otimes V)$ with (commutative) multiplication $$(v,v',u\otimes u')(w,w',z\otimes z')=(0,0,v\otimes w' + v'\otimes w).$$  
The involution interchanging $0$ and $1$ switches the two $V$ factors and acts on $V\otimes V$ by sending $v\otimes w$ to $w\otimes v$.  

Below we indicate the results of calculating  $H_{k,1}(\hairy\assoc_V)$ for some small values of $k$.

For $k=1$, we have $H_{1,1}(\hairy\assoc_V)=\HD_0(I\assoc_V)= (V\oplus V\oplus (V\otimes V))/(1-t)(1-s).$  
Since $I\assoc_V$ is commutative, $t$ acts trivially and we need only factor out by the action of the involution $y$.   
Since this interchanges the first two $V$ factors and acts on $V\otimes V$ by  $v\otimes w \leftrightarrow w\otimes v$, we obtain
$$H_{1,1}(\hairy\assoc_V)=HD_0(I\assoc_V)\iso V\oplus \ext^2 V.$$  In terms of hairy graphs, this is generated by
$$
\begin{minipage}{2cm}
\begin{tikzpicture}[scale=.7]
\node (v) [right] at (.5,0) {$v$};
\draw (0,0) -- (v);
\draw [->] (0,0) to [ out=90, in=90] (-2,0);
\draw (-2,0) to [ out=-90, in=-90] (0,0);
 \draw  (0,0) circle (.2);
\end{tikzpicture}
\end{minipage}
\quad \hbox{and} \quad
\begin{minipage}{3cm}
\begin{tikzpicture}[scale=.7]
\node (w) [left] at (-.5,0) {$w$};
\node (v) [right] at (.5,0) {$v$};
\draw (w) to (0,0) to (v);
\draw [->] (0,0) to [ out=90, in=90] (-2,0);
\draw (-2,0) to [ out=-90, in=-90] (0,0);
 \draw  (0,0) circle (.2);
\end{tikzpicture}
\end{minipage}
$$

For $k=2$ we obtain
$$H_{2,1}(\hairy\assoc_V)=HD_1(I\assoc_V) \iso \ext^2 V\oplus \ext^2 V\oplus (V\otimes \ext^2 V),$$
where the first summand is generated by
 $$
 \begin{tikzpicture}[scale=.7]
\begin{scope}[decoration={markings,mark = at position 0.5 with {\arrow{stealth}}}] 
\node (v) [right] at (.5,0) {$v$};
\node (w) [left] at (-2.5,0) {$w$};
\draw (0,0) -- (v);
\draw (-2,0) -- (w);
\midarrow (0,0) to [ out=90, in=90] (-2,0);
\midarrow (-2,0) to [ out=-90, in=-90] (0,0);
 \draw  (0,0) circle (.2);
  \draw  (-2,0) circle (.2);
\end{scope}
\end{tikzpicture}
$$
the second by 
$$
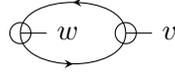
\begin{tikzpicture}[scale=.7]
\begin{scope}[decoration={markings,mark = at position 0.5 with {\arrow{stealth}}}] 
\node (v) [right] at (.5,0) {$v$};
\node (w) [right] at (-1.5,0) {$w$};
\draw (0,0) -- (v);
\draw (-2,0) -- (w);
\midarrow (0,0) to [ out=90, in=90] (-2,0);
\midarrow (-2,0) to [ out=-90, in=-90] (0,0);
 \draw  (0,0) circle (.2);
  \draw  (-2,0) circle (.2);
\end{scope}
\end{tikzpicture}
$$
and the third by the difference
$$
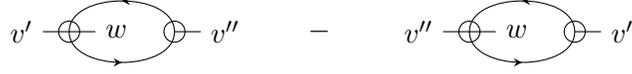
\begin{minipage}{4cm}
\begin{tikzpicture}[scale=.7]
\begin{scope}[decoration={markings,mark = at position 0.5 with {\arrow{stealth}}}] 
\node (v) [right] at (.5,0) {$v^{\prime\prime}$};
\node (w) [right] at (-1.5,0) {$w$};
\node (v2) [left] at (-2.5,0) {$v'$};
\draw (0,0) -- (v);
\draw (-2,0) -- (w);
\draw (-2,0) -- (v2);
\midarrow (0,0) to [ out=90, in=90] (-2,0);
\midarrow (-2,0) to [ out=-90, in=-90] (0,0);
 \draw  (0,0) circle (.2);
  \draw  (-2,0) circle (.2);
\end{scope}
\end{tikzpicture}
\end{minipage}
- \qquad
 \begin{minipage}{4cm}
\begin{tikzpicture}[scale=.7]
\begin{scope}[decoration={markings,mark = at position 0.5 with {\arrow{stealth}}}] 
\node (v) [right] at (.5,0) {$v^{\prime}$};
\node (w) [right] at (-1.5,0) {$w$};
\node (v2) [left] at (-2.5,0) {$v^{\prime\prime}$};
\draw (0,0) -- (v);
\draw (-2,0) -- (w);
\draw (-2,0) -- (v2);
\midarrow (0,0) to [ out=90, in=90] (-2,0);
\midarrow (-2,0) to [ out=-90, in=-90] (0,0);
 \draw  (0,0) circle (.2);
  \draw  (-2,0) circle (.2);
\end{scope}
\end{tikzpicture}
\end{minipage}
$$

For $k=3$ computations give the following result:
$$H_{3,1}(\hairy\assoc_V)=HD_2(I\assoc_V)\iso  \left(V^{\otimes 3}\right)_{\Z_3}\oplus \left(V\otimes \ext^2 V\right)\oplus \left(V\otimes (V^{\otimes 3}\right)_{\Z_3})\oplus \left(\ext^4V\right). $$
We summarize our computations  in the following table, where the results are decomposed by degree and the $(k,d)$ entry gives the partitions $\lambda$ which appear in  $HD_k(I\assoc_V)\degree{d}.$
$$
\begin{array}{c|cccc}
\text{ \tiny k}\backslash^d&1&2&3&4\\
\hline
0&[1]&[1^2]&0&0\\
1&0&2 [1^2]&[2,1]\oplus[1^3]&0\\
2&0&0&[3]\oplus[2,1]\oplus 2[1^3]& [4]\oplus[3,1]\oplus[2,1^2]\oplus2[1^4]
\end{array}
$$

\section{Homology classes for mapping class groups and $\Out(F_n)$}\label{sec:implications}
The assembly map allows us to take hairy graph homology classes of small rank and glue them together to form an ordinary graph cohomology class. To date, this has been succesfully used to build homology classes for $\Aut(F_n),\Out(F_n)$ and $\operatorname{Mod}(g,n)$ by assembling  homology classes from $H_1(\hairy_V).$   Indeed all known rational homology classes for $\Aut(F_n)$ and $\Out(F_n)$ can be assembled from first homology classes represented by hairy graphs of ranks $1$ and $2$.

For $\cO=\assoc$ the first homology of the hairy graph complex is isomorphic to $\ext^3V$ in rank $0,$  to $V\oplus \ext^2V$ in rank $1$ and vanishes in higher rank (\cite{CKV}). The rank $1$  
classes corresponding to $\ext^2 V$ can be glued together in a circle to get ``ornate necklaces;"  
 which represent nontrivial cohomology classes for mapping class groups of punctured tori \cite{Jim}.  It is natural to ask whether higher dimensional hairy graph classes can also be assembled to give new unstable homology for mapping class groups.  There are many other known constructions of unstable classes in the homology of mapping class groups; can some or all of these be assembled from higher hairy graph cohomology classes of smaller rank?   
 
For $\cO=\lie$ gluing rank $1$ and rank $2$ hairy first cohomology classes together gives all known rational homology for $\Out(F_n)$. The above calculations show that there is plenty of higher hairy graph homology in rank $1$, and in \cite{CHKV} we will show that there is also plenty of higher hairy graph homology in rank $2$. Assembling these gives a host of   potential classes in the unstable homology of  $\Out(F_n)$.

\end{document}